\documentclass[11pt]{amsart}
\usepackage[margin=1in]{geometry}

\usepackage{amssymb}
\usepackage{amsthm}
\usepackage{amsmath}
\usepackage{mathrsfs}
\usepackage{amsbsy}
\usepackage{bm}
\usepackage{hyperref}
\usepackage{tikz}
\usepackage{array}
\usepackage{enumerate}
\usepackage{bbm}
\usepackage{comment}
\usepackage{mathtools}
\usepackage{tabu}
\usepackage{makecell} 
\usepackage{colortbl}
\usepackage{xcolor}

\DeclareFontFamily{U}{mathx}{}
\DeclareFontShape{U}{mathx}{m}{n}{<-> mathx10}{}
\DeclareSymbolFont{mathx}{U}{mathx}{m}{n}
\DeclareMathAccent{\widecheck}{0}{mathx}{"71}

\definecolor{LightBlue}{rgb}{0,0.8,1} 
\definecolor{Green}{rgb}{0,0.863,0}
\definecolor{DarkGreen}{rgb}{0,0.5,0}
\definecolor{MildGreen}{rgb}{0,0.784,0}
\definecolor{Turquoise}{rgb}{0,0.68,0.38}
\definecolor{NormalGreen}{rgb}{0,0.8,0}
\definecolor{LightGreen}{rgb}{0,0.922,0}
\definecolor{Magenta}{rgb}{0.784,0,0.784}
\definecolor{Yellow}{rgb}{0.95,0.95,0}
\definecolor{lavender}{rgb}{0.4,0,1}
\definecolor{peach}{rgb}{1,0.43,0.39} 
\definecolor{DarkPink}{rgb}{1,0,0.45} 
\definecolor{NewBlue}{rgb}{0,0.3,0.8}
\definecolor{Teal}{rgb}{0,0.784,0.784}
\definecolor{Gold}{rgb}{0.929,0.784,0.392}

\hypersetup{colorlinks=true, citecolor=LightBlue, linkcolor=NewBlue,urlcolor=NewBlue}

\usepackage{hhline}
\allowdisplaybreaks
\usepackage[noadjust]{cite}

\usepackage{caption}
\usepackage[noabbrev,capitalise,nameinlink]{cleveref}
\crefname{conjecture}{Conjecture}{Conjectures}

\newtheorem{theorem}{Theorem}[section]
\newtheorem{proposition}[theorem]{Proposition}

\theoremstyle{definition}

\newtheorem{remark}[theorem]{Remark}
\newtheorem{example}[theorem]{Example} 

\usepackage{etoolbox}

\newcommand{\includeSymbol}[1]{\ensuremath{%
	\mathchoice
		{\raisebox{-.4mm}{\includegraphics[height=2.1ex]{#1}}}	
		{\raisebox{-.4mm}{\includegraphics[height=2.1ex]{#1}}}
		{\raisebox{-.3mm}{\includegraphics[height=1.6ex]{#1}}}
		{\raisebox{-.2mm}{\includegraphics[height=1ex]{#1}}}
}} 

\robustify{\includeSymbol}
\newcommand{\Paral}{\includeSymbol{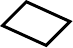}} 

\robustify{\includeSymbol}
\newcommand{\Rect}{\includeSymbol{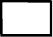}}

\robustify{\includeSymbol}
\newcommand{\GoldBump}{\includeSymbol{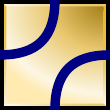}}

\robustify{\includeSymbol}
\newcommand{\Bump}{\includeSymbol{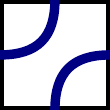}}

\robustify{\includeSymbol}
\newcommand{\Cross}{\includeSymbol{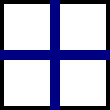}}

\robustify{\includeSymbol}
\newcommand{\SEElbow}{\includeSymbol{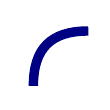}}

\robustify{\includeSymbol}
\newcommand{\NWElbow}{\includeSymbol{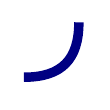}}

\robustify{\includeSymbol}
\newcommand{\LLT}{\includeSymbol{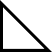}} 

\robustify{\includeSymbol}
\newcommand{\URT}{\includeSymbol{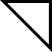}} 

\robustify{\includeSymbol}
\newcommand{\ULT}{\includeSymbol{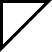}} 

\robustify{\includeSymbol}
\newcommand{\Le}{\includeSymbol{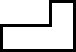}}

\robustify{\includeSymbol}
\newcommand{\Hole}{\includeSymbol{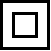}}

\newcommand{\SSS}{\mathfrak{S}}

\newcommand{\h}{\mathbf{H}}

\newcommand{\bb}{\mathfrak{b}}

\newcommand{\ZZ}{\mathbb{Z}}
\newcommand{\sub}{\mathrm{sub}}
\newcommand{\bw}{\mathbf{w}} 
\newcommand{\sS}{\mathscr{S}}
\newcommand{\cc}{\mathsf{c}}
\newcommand{\ff}{\mathsf{f}}
\newcommand{\hh}{\mathsf{h}}

\newcommand{\xx}{x} 
\newcommand{\yy}{y} 
\newcommand{\vv}{\mathsf{v}} 
\newcommand{\RR}{\mathbf{R}}
\newcommand{\PP}{\mathscr{P}}
\newcommand{\DD}{\mathrm{D}}

\newcommand{\dfn}[1]{\textcolor{Turquoise}{\emph{#1}}}

\begin{document}

\title[]{Permutons from Demazure Products} 
\subjclass[2010]{}

\author[]{Colin Defant}
\address[]{Department of Mathematics, Harvard University, Cambridge, MA 02138, USA}
\email{colindefant@gmail.com}

\begin{abstract}
We construct and analyze several new families of permutons arising from random processes involving the Demazure product on the symmetric group. First, we consider Demazure products associated to random pipe dreams, generalizing the Grothendieck permutons introduced by Morales, Panova, Petrov, and Yeliussizov by replacing staircase shapes with arbitrary order-convex shapes. Using the totally asymmetric simple exclusion process (TASEP) with geometric jumps, we prove precise scaling limit and fluctuation results for the associated height functions, showing that these models belong to the Kardar--Parisi--Zhang (KPZ) universality class. We then consider permutons obtained by applying deterministic sequences of bubble-sort operators to random initial permutations. We again provide precise descriptions of the limiting permutons. In a special case, we deduce the exact forms of the standard bubble-sort permutons, the supports of which were computed by DiFranco. 

A crucial tool in our analysis is a formulation, due to Chan and Pflueger, of the Demazure product as matrix multiplication in the min-plus tropical semiring. This allows us to define a Demazure product on the set of permutons. We discuss further applications of this product. For instance, we show that the number of inversions of the Demazure product of two independent uniformly random permutations of size $n$ is $\binom{n}{2}(1-o(1))$.  
\end{abstract} 

\maketitle

\section{Introduction}\label{sec:intro} 

\subsection{Permutons} 
A \dfn{permuton} is a probability measure $\mu$ on the unit square $[0,1]^2$ that has uniform marginals in the sense that $\mu([0,1]\times [a,b])=\mu([a,b]\times [0,1])=b-a$ for all $0\leq a\leq b\leq 1$. The \dfn{plot} of a permutation $u$ in the symmetric group $\SSS_n$ is the diagram showing the points $(i/n,u(i)/n)$ for all $i\in[n]\coloneq\{1,\ldots,n\}$. We can enlarge each point $(i/n,u(i)/n)$ to a square $\Box_i(u)$ consisting of the points $(x,y)\in[0,1]^2$ such that $i-1\leq nx\leq i$ and $u(i)-1\leq ny\leq u(i)$; the resulting arrangement of squares forms a permuton $\pi_u$ in which each of the squares $\Box_i(u)$ has constant density $n$ and the complement of these squares has constant density $0$. 

Permutons have received a great deal of attention over the past fifteen years because they serve as scaling limits of large (often random) permutations. More precisely, we can think of a permuton $\mu$ as a scaling limit of a sequence $(u_n)_{n\geq 1}$ with $u_n\in\SSS_n$ if $(\pi_{u_n})_{n\geq 1}$ converges weakly to $\mu$. According to \cite{HKMRS}, the sequence $(\pi_{u_n})_{n\geq 1}$ converges weakly to $\mu$ if and only if all of the pattern densities in $u_n$ converge to the corresponding pattern densities in $\mu$. 

An alternative---and often more convenient---way of encoding a permuton $\mu$ is provided by its \dfn{height function} $\h_\mu\colon[0,1]^2\to[0,1]$, which is defined by 
\begin{equation}\label{eq:height}
\h_\mu(x,y)=\mu([x,1]\times [0,y]). 
\end{equation}
A sequence $(\mu_n)_{n\geq 1}$ of permutons converges weakly to a permuton $\mu$ if and only if the sequence $(\h_{\mu_n})_{n\geq 1}$ converges pointwise to $\h_\mu$. Given a permutation $u$, we write $\h_u$ instead of $\h_{\pi_u}$.  

\subsection{The Demazure Product} 

Let $s_i$ denote the transposition $(i\,\, i+1)$ in $\SSS_n$. An \dfn{inversion} of a permutation $v\in\SSS_n$ is a pair $(a,b)$ such that $a<b$ and $v^{-1}(a)>v^{-1}(b)$. Let $\ell(v)$ denote the number of inversions of $v$. A \dfn{word} is a tuple $(i_1,\ldots,i_r)$ of elements of $[n-1]$; we say this word \dfn{represents} the permutation $s_{i_1}\cdots s_{i_r}$. A word representing a permutation $u$ is \dfn{reduced} if it has the minimum length among all words representing $u$; the length of a reduced word for $u$ is $\ell(u)$. We write $\mathbf{u}\mathbf{v}$ for the concatenation of two words $\mathbf{u}$ and $\mathbf{v}$. One can apply a \dfn{commutation move} to a word by swapping two adjacent letters whose values differ by at least $2$. Two words are \dfn{commutation equivalent} if one can be obtained from the other via a sequence of commutation moves. 

For $i\in[n-1]$, define $\tau_i\colon\SSS_n\to\SSS_n$ by 
\begin{equation}\label{eq:tau}
\tau_i(u)=\begin{cases}
    us_i & \text{if } \ell(us_i)>\ell(u); \\
    u & \text{if } \ell(us_i)<\ell(u).
\end{cases}
\end{equation} 
In other words, $\tau_i$ acts on a permutation by putting the entries in positions $i$ and $i+1$ in decreasing order (leaving the rest of the permutation unchanged). 

The \dfn{Demazure product} is the unique associative product $\star$ on $\SSS_n$ that satisfies $u\star s_i=\tau_i(u)$ for all $u\in\SSS_n$ and $i\in[n-1]$. We define the Demazure product of a word $(i_1,\ldots,i_r)$ to be $s_{i_1}\star\cdots\star s_{i_r}$. We can compute the Demazure product $u\star v$  of two permutations $u,v\in\SSS_n$ by first choosing a reduced word $\mathbf{u}$ for $u$ and a reduced word $\mathbf{v}$ for $v$ and then computing the Demazure product of the word $\mathbf{u}\mathbf{v}$. The set $\SSS_n$ equipped with the Demazure product is known as the \dfn{$0$-Hecke monoid} (of $\SSS_n$) \cite{CP, Demazure, KnutsonMiller, LORYY, Pflueger}. Our goal in this article is to study various permutons that arise from natural random processes involving the Demazure product. 

\subsection{Shapes and Pipe Dreams} 
For $x,y\in\ZZ$ with $x>y$, define the \dfn{box} $\bb(x,y)$ to be the axis-parallel unit square in $\mathbb{R}^2$ whose lower-left vertex is $(x,y)$. The \dfn{content} of a box $\bb(x,y)$ is $x-y$. A set of boxes is called a \dfn{shape}. A shape is naturally equipped with a partial order $\leq$ in which ${\bb(x,y)\leq\bb(x',y')}$ if and only if $x\leq x'$ and $y\leq y'$. Given a shape $\mathscr{S}$, we can choose a linear extension $\bb(x_1,y_1),\bb(x_2,y_2),\ldots,\bb(x_r,y_r)$ of $\sS$ to obtain a word \[\mathbf{w}(\mathscr{S})=(x_1-y_1,x_2-y_2,\ldots,x_r-y_r).\] Any two words arising in this manner from the same shape are commutation equivalent, which makes them effectively the same for our purposes. 

For $a\in\mathbb R$, we define the line \[\mathfrak{L}_a=\{(x,y)\in\mathbb{R}^2:x-y=a\}.\] 
For each positive integer $n$, let $\boldsymbol{\Lambda}_n$ be the set of lattice paths that use unit east steps and unit south steps, start on $\mathfrak{L}_0$, and end on $\mathfrak{L}_n$. Suppose $\Lambda_{\swarrow}$ and $\Lambda_{\nearrow}$ are two lattice paths in $\boldsymbol{\Lambda}_n$ such that $\Lambda_{\swarrow}$ stays weakly southwest of $\Lambda_{\nearrow}$. We obtain a shape $\sS(\Lambda_{\swarrow},\Lambda_{\nearrow})$ consisting of all boxes lying in the region bounded by the paths $\Lambda_{\swarrow}$ and $\Lambda_{\nearrow}$ and the lines $\mathfrak{L}_0$ and $\mathfrak{L}_n$. A shape obtained in this manner is called \dfn{order-convex}. 

A subword of $\mathbf{w}(\mathscr{S})$ can be represented graphically via a \dfn{pipe dream}. For simplicity, let us assume for now that $\sS$ is order-convex so that $\sS=\sS(\Lambda_{\swarrow},\Lambda_{\nearrow})$ for some positive integer $n$ and some lattice paths $\Lambda_{\swarrow},\Lambda_{\nearrow}\in\boldsymbol{\Lambda}_n$. Let $\bb(x_1,y_1),\bb(x_2,y_2),\ldots,\bb(x_r,y_r)$ be the linear extension of $\mathscr{S}$ that induces the word $\mathbf{w}(\mathscr{S})$. If the $i$-th letter in $\mathbf{w}(\mathscr{S})$ is deleted when we form the subword of $\bw(\sS)$, fill the box $\bb(x_i,y_i)$ with a \dfn{bump tile} $\Bump$\,; otherwise, fill $\bb(x_i,y_i)$ with a \dfn{cross tile} $\Cross$\,. Let us also add \dfn{elbows} $\SEElbow$ and $\NWElbow$ along the lines $\mathfrak{L}_{1/2}$ and $\mathfrak{L}_{n-1/2}$, respectively so that we have $n$ \dfn{pipes} that travel from $\Lambda_{\swarrow}$ to $\Lambda_{\nearrow}$. See the left-hand side of \cref{fig:pipes_1}. 

\begin{figure}[]
  \begin{center}
  \includegraphics[height=6.375cm]{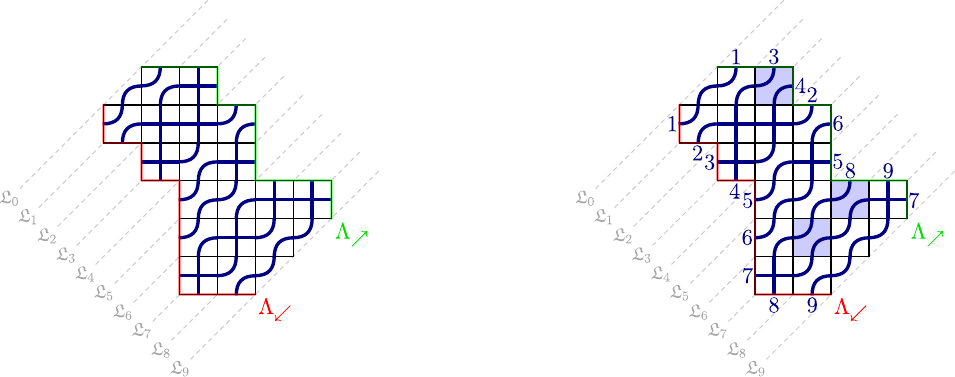}
  \end{center}
\caption{On the left is a pipe dream in an order-convex shape $\sS(\Lambda_{\swarrow},\Lambda_{\nearrow})$, where the lattice paths $\Lambda_{\swarrow},\Lambda_{\nearrow}\in\boldsymbol{\Lambda}_9$ are drawn in {\color{red}red} and {\color{Green}green}. The image on the right computes the Demazure product of the corresponding subword by resolving the three crossings in the shaded boxes. This Demazure product is $134265897$. }\label{fig:pipes_1} 
\end{figure}

We can also compute the Demazure product of a subword $\mathbf{v}$ of $\bw(\sS)$ from the associated pipe dream as follows. First, we label the pipes $1,\ldots,n$ in the order that they appear along $\Lambda_\swarrow$ when we read from northwest to southeast (so pipe $i$ starts on the line $\mathfrak{L}_{i-1/2}$). We then trace along the pipes from southwest to northeast. Whenever we see two pipes that have already crossed and are about to cross each other again in some box $\bb(x,y)$, we change the cross tile in $\bb(x,y)$ into a bump tile, and we say that we \dfn{resolved the crossing} in box $\bb(x,y)$. After continuing this process throughout the entire shape, we are left with a pipe dream that is \dfn{reduced} in the sense that no two pipes cross each other more than once. We then read the labels of the pipes along $\Lambda_{\nearrow}$ from northwest to southeast to obtain a permutation in $\SSS_n$. This permutation is the Demazure product of $\mathbf{v}$. See the right-hand side of \cref{fig:pipes_1}.  

\subsection{Random Subwords and Pipe Dreams}\label{subsec:intro_MPPY}

Now fix a probability $p\in(0,1]$. Consider an order-convex shape ${\mathscr{S}=\sS(\Lambda_{\swarrow},\Lambda_{\nearrow})}$. Let $\sub_p(\mathbf{w}(\mathscr{S}))$ be the random subword of $\mathbf{w}(\mathscr{S})$ obtained by independently deleting each letter with probability $1-p$. Let $\Delta_p(\mathscr{S})$ be the Demazure product of the word $\sub_p(\mathbf{w}(\mathscr{S}))$. We are interested in studying the random permutation $\Delta_p(\mathscr{S})$, whose distribution depends only on the shape $\mathscr{S}$ and the probability $p$ (not the particular choice of the word $\mathbf{w}(\mathscr{S})$ associated to $\mathscr{S}$). As above, the subword $\sub_p(\bw(\sS))$ can be viewed as a random pipe dream in $\sS$, and the Demazure product $\Delta_p(\sS)$ can be computed from the pipe dream by resolving crossings. See \cref{fig:rectangle_example}.   

\begin{figure}[]
  \begin{center}
  \includegraphics[width=\linewidth]{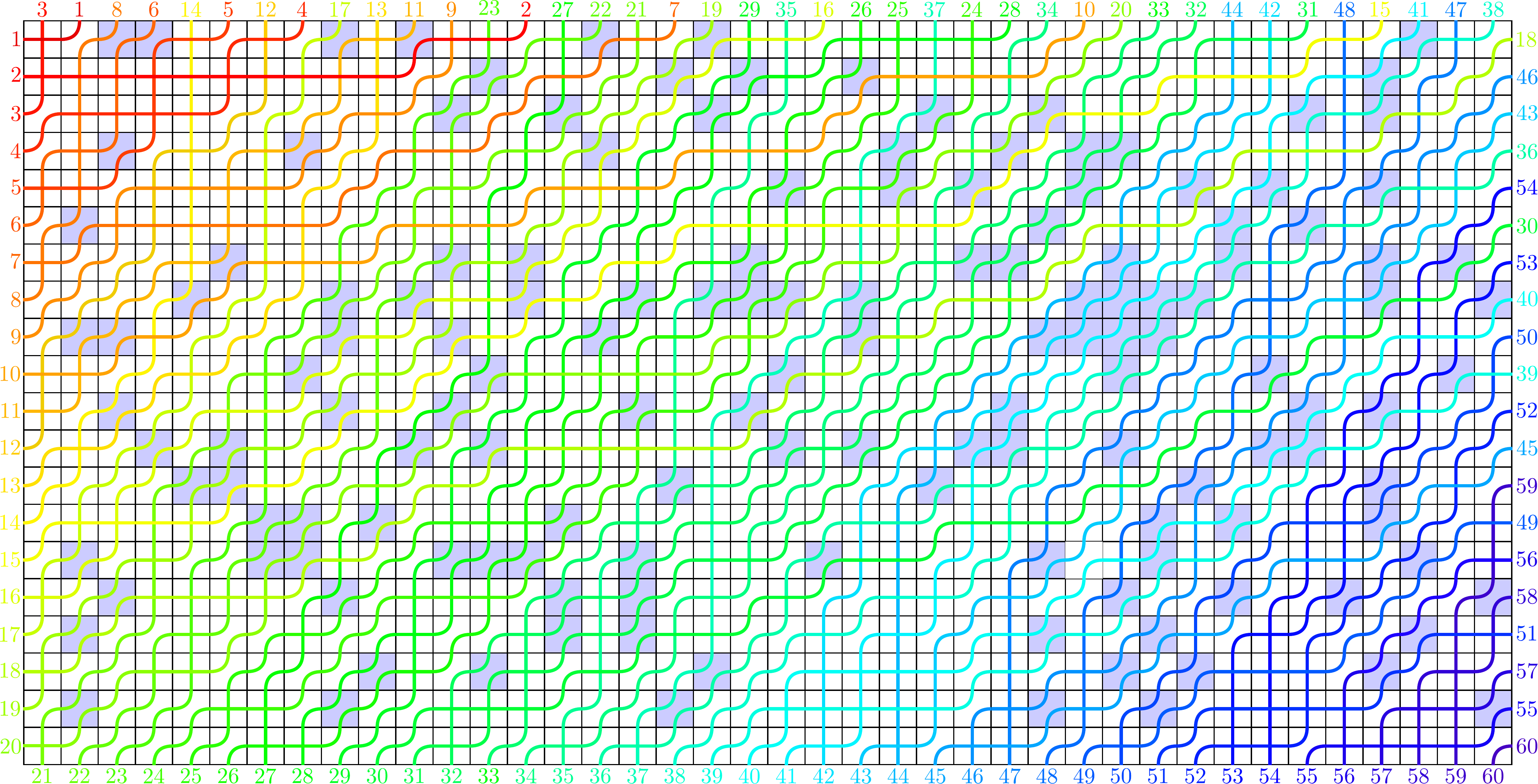}
  \end{center}
\caption{A random pipe dream with parameter $p=1/2$ in a $20\times 40$ rectangle shape $\mathscr{S}$ with crossings resolved in the shaded boxes. Different pipes receive different colors. The Demazure product $\Delta_p(\sS)$ is obtained by reading the numbers along the northeast boundary.  }\label{fig:rectangle_example} 
\end{figure} 

\begin{figure}[]
  \begin{center}
  \includegraphics[height=5cm]{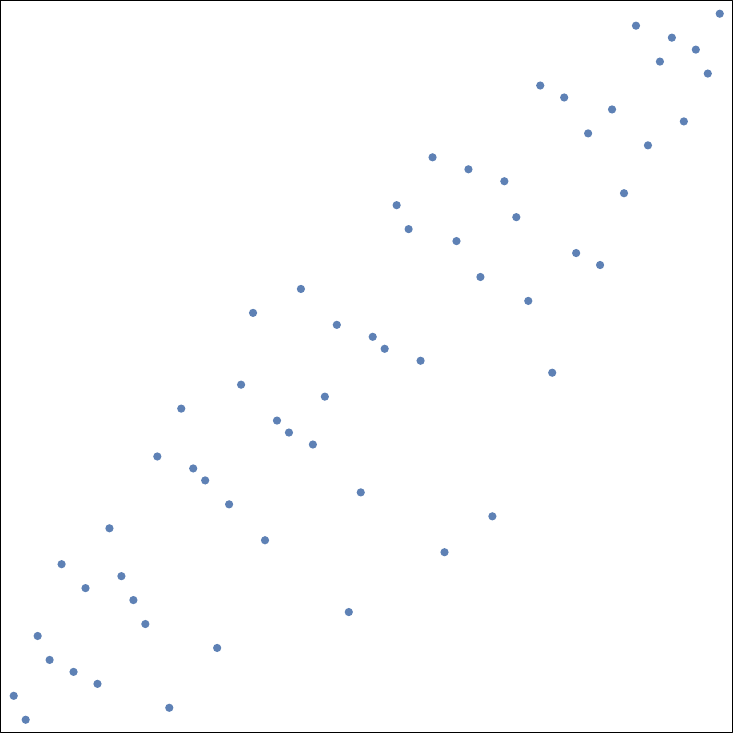}\quad\includegraphics[height=5cm]{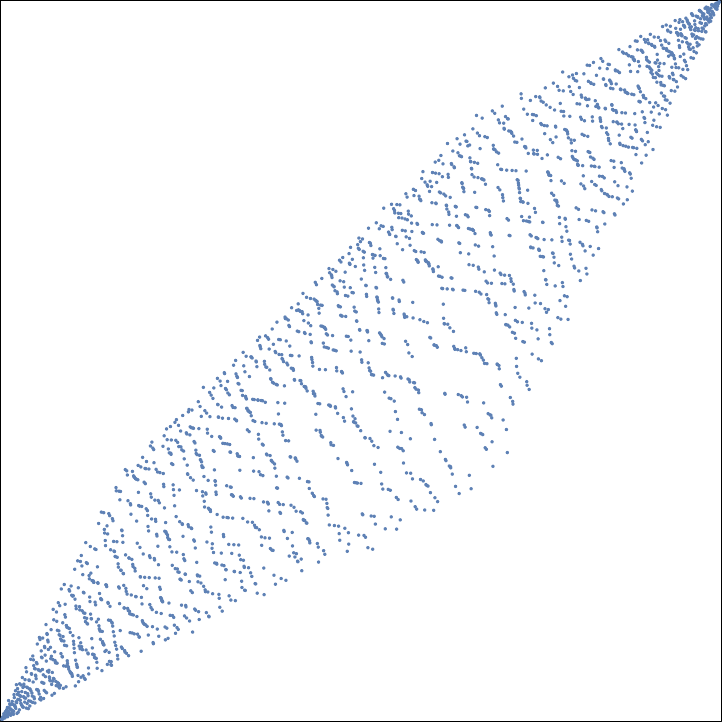}\quad\includegraphics[height=5cm]{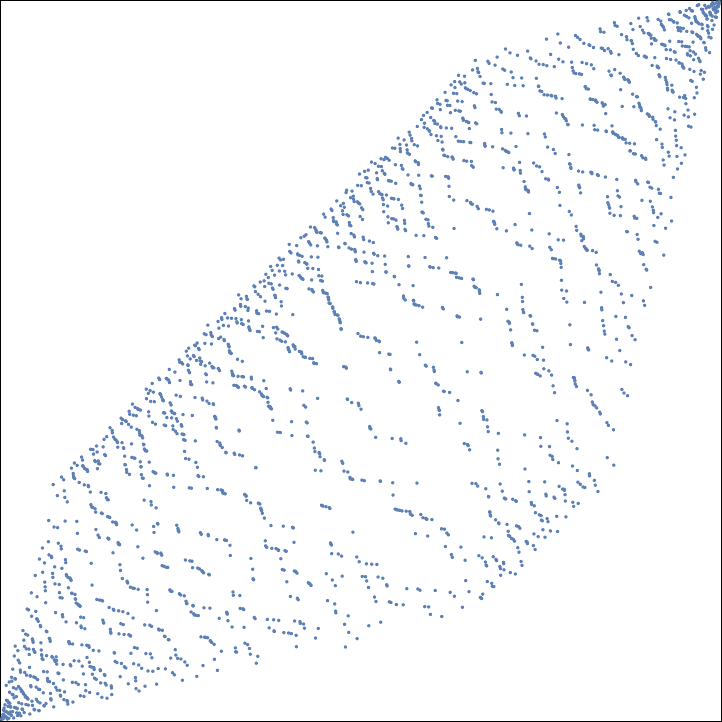}
  \end{center}
\caption{On the left is the plot of the random permutation $\Delta_{1/2}(\sS)\in\SSS_{60}$ from \cref{fig:rectangle_example}. In the middle and on the right are plots of the random permutations $\Delta_{1/2}(\sS)$ and $\Delta_{3/4}(\sS)$ in $\SSS_{2400}$, where $\sS$ is an $800\times 1600$ rectangle shape.  }\label{fig:rectangle_example_plot}  
\end{figure} 

One especially noteworthy shape is the \dfn{staircase shape} $\sS_{\mathrm{stair}}^{(n)}$, which is the order-convex shape whose southwest boundary is the path that travels south from $(0,0)$ to $(0,-n)$ and whose northeast boundary is the path that travels east from $(0,0)$ to $(n,0)$. Pipe dreams in the staircase shape are ubiquitous in Schubert calculus, where they are used to compute Schubert polynomials and Grothendieck polynomials \cite{BB,FominKirillov,HPW,KnutsonMiller2,KnutsonMiller}. Morales, Panova, Petrov, and Yeliussizov \cite{MPPY} studied the Demazure product $\Delta_p(\sS_{\mathrm{stair}}^{(n)})$, showing that it converges with high probability to a particular permuton as $n\to\infty$. Their approach connected random pipe dreams in the staircase shape with the stochastic $6$-vertex model and the TASEP with geometric jumps. This allowed them to employ asymptotic results about the TASEP to derive very precise asymptotic results about $\Delta_p(\sS_{\mathrm{stair}}^{(n)})$. 

Our first main result is a vast generalization of one of the theorems from \cite{MPPY} that applies to arbitrary order-convex shapes rather than just staircase shapes. In contrast to that article, we bypass the stochastic $6$-vertex model and work directly with the TASEP with geometric jumps. In order to state this result, we need some additional notation and terminology. 

Let $\mathfrak B$ be the set of continuous and piecewise smooth functions $\theta\colon[0,1]\to\mathbb R$ such that the derivative of $\theta$ lies in $[0,1]$ at every point where it exists. 
Given a function $\theta\in\mathfrak B$, we form a parametric curve ${C_\theta=\{C_\theta(z):0\leq z\leq 1\}}$ in $\mathbb R^2$, where $C_\theta(z)=(\theta(z),\theta(z)-z)$. Let \[\mathfrak{C}=\{C_\theta:\theta\in\mathfrak{B}\}.\] A curve in $\mathfrak{C}$ must start on the line $\mathfrak{L}_0$ and end on the line $\mathfrak{L}_1$. 
If $\Lambda$ is a lattice path in $\boldsymbol{\Lambda}_n$, then we can view the scaled path $\frac{1}{n}\Lambda$ as a curve in $\mathfrak C$; we then write $\theta_{\Lambda}$ for the function in $\mathfrak B$ such that $\frac{1}{n}\Lambda=C_{\theta_{\Lambda}}$. Given $\varphi,\psi\in\mathfrak B$ such that $\varphi(z)\leq\psi(z)$ for all $z\in[0,1]$, we obtain a region $\mathscr{R}^{\varphi,\psi}\subseteq\mathbb R^2$ bounded by the curves $C_\varphi$ and $C_\psi$ and the lines $\mathfrak L_0$ and $\mathfrak L_1$. Let us write 
\begin{equation}\label{eq:RR}
\RR=\{\mathscr{R}^{\varphi,\psi}:\varphi,\psi\in\mathfrak{B},\,\varphi(z)\leq\psi(z)\text{ for all }z\in[0,1]\}
\end{equation} 
for the set of all regions of this form.

Fix $\varphi,\psi\in\mathfrak B$ such that $\varphi(z)\leq \psi(z)$ for all $z\in[0,1]$. Define functions $\ff_p^{\varphi,\psi}\colon[0,1]\to\mathbb R$ and $\hh_p^{\varphi,\psi}\colon[0,1]\to\mathbb R$ by 
\begin{equation}\label{eq:f}
\ff_p^{\varphi,\psi}(x,y)=\frac{1}{p}\left[(2-p)(\psi(x)-\varphi(y))+y-x-2\sqrt{(1-p)(\psi(x)-\varphi(y))(\psi(x)-\varphi(y)-x+y)}\right]
\end{equation} 
and 
\begin{equation}\label{eq:h} 
\hh_p^{\varphi,\psi}(\xx,\yy)=\min\{\max\{0,y-x,\ff_p^{\varphi,\psi}(x,y)\},1-x,y\}. 
\end{equation} 

The next theorem states that if we have a sequence of order-convex shapes that, after appropriate scaling, converge to a region $\DD\in\RR$, then the Demazure products of random pipe dreams in these shapes converge to a particular permuton $\zeta_p^\DD$ with high probability. 

\begin{theorem}\label{thm:NewMPPY}
Fix $\varphi,\psi\in\mathfrak B$ such that $\varphi(z)\leq\psi(z)$ for all $z\in[0,1]$, and let $\DD=\mathscr{R}^{\varphi,\psi}\in\RR$. There is a permuton $\zeta_p^\DD$ whose height function is $\hh_p^{\varphi,\psi}$. For each $n\geq 1$, choose lattice paths $\Lambda_{\swarrow}^{(n)}$ and $\Lambda_{\nearrow}^{(n)}$ in $\boldsymbol{\Lambda}_n$ such that the former lies weakly below the latter. Assume that $\theta_{\Lambda_{\swarrow}^{(n)}}$ and $\theta_{\Lambda_{\nearrow}^{(n)}}$ converge pointwise to $\varphi$ and $\psi$, respectively. Let \[u_n=\Delta_p\left(\sS\left(\Lambda_{\swarrow}^{(n)},\Lambda_{\nearrow}^{(n)}\right)\right).\] Then \[\lim_{n\to\infty}\h_{u_n}(\xx,\yy)=\hh_p^{\varphi,\psi}(\xx,\yy)\] with probability $1$ for every $(x,y)\in[0,1]^2$. Equivalently, $(\pi_{u_n})_{n\geq 1}$ converges weakly to $\zeta_p^\DD$.  
\end{theorem} 

When $p=1$, there is no randomness involved in \cref{thm:NewMPPY}. However, when $0<p<1$, the permutations $u_n$ in that theorem are stochastic, and we can consider height function fluctuations. The next theorem describes these fluctuations asymptotically, showing that this model belongs to the famous Kardar--Parisi--Zhang (KPZ) universality class (see the survey \cite{Corwin} for more information). Let us define 
\begin{equation}\label{eq:K}
\mathscr{K}^{\varphi,\psi}=\{(x,y)\in[0,1]^2:\max\{0,y-x\}<\ff_p^{\varphi,\psi}(x,y)<\min\{1-x,y\}\}.
\end{equation} 
For ${\sf m},{\sf t}\in\mathbb R$, define 
\begin{equation}\label{eq:v} 
\vv_p({\sf m},{\sf t})=\frac{{\sf m}^{1/3}\left(\sqrt{{\sf t}/p}-\sqrt{{\sf m}}\right)^{2/3}\left(\sqrt{p{\sf t}}-\sqrt{{\sf m}}\right)^{2/3}}{\left(\sqrt{{\sf t}}-\sqrt{p{\sf m}}\right){\sf t}^{1/6}}. 
\end{equation}

\begin{theorem}\label{thm:fluctuations} 
Preserve the notation from \cref{thm:NewMPPY}. If $(\xx,\yy)\in\mathscr K^{\varphi,\psi}$, then for all $r\in\mathbb R$, we have
\[\lim_{n\to\infty}\mathbb P\left(\frac{\h_{u_n}(\xx,\yy)-\hh_p^{\varphi,\psi}(\xx,\yy)}{\vv_p\left(x+y-\hh_p^{\varphi,\psi}(x,y),\psi(x)-\varphi(y)\right)n^{-2/3}}\geq -r\right)=F_2(r),\] where $F_2$ is the cumulative distribution function of the Tracy--Widom GUE distribution. 
\end{theorem}

When $\varphi(z)=0$ and $\psi(z)=z$ for all $z\in[0,1]$, \cref{thm:NewMPPY,thm:fluctuations} recover a theorem from \cite{MPPY}. 

For a particular (reasonably nice) choice of $\DD$, one can describe the geometric properties of the permuton limit $\zeta_p^{\DD}$ from \cref{thm:NewMPPY} quite explicitly. We will illustrate this with some examples in \cref{subsec:gems,subsec:malts}. In \cref{subsec:gems}, we consider the case where $\DD$ is a rectangle. This leads to a new class of permutons that we call \emph{peridot permutons}; see \cref{fig:rectangle_example_plot,fig:gems}. In \cref{subsec:malts}, we consider the case where $\varphi$ and $\psi$ are linear functions so that $\DD$ is a trapezoid. This leads to a new class of permutons that we call \emph{Polyphemus permutons}; see \cref{fig:trapezoid_example_plot}.  

We can actually generalize the permuton limit result in \cref{thm:NewMPPY} further by considering even more exotic shapes. Roughly speaking, one could use our techniques to compute permuton limits whenever the scaling limits of the shapes can be ``decomposed'' into finitely many disjoint pieces in $\RR$. However, in practice, the resulting formulas become complicated quickly. In \cref{subsec:penguins}, we will discuss this method and provide examples illustrating some of the permuton shapes that it can produce (\emph{Platyhelminthes permutons} and \emph{pointy peanut permutons}).  

\begin{remark}
Here, we consider the Demazure product of a random subword of a fixed starting word associated to a shape. In \cite{DefantStoned}, the author studied a similar procedure, but in an affine Weyl group rather than the symmetric group.   
\end{remark}

\subsection{Bubble-Sort Permutons}\label{subsec:intro_Bubble} 

In the previous subsection, we considered a random permutation obtained by taking the Demazure product of a random subword of a fixed starting word. We now flip the randomness on its head and apply a deterministic sequence of operators to a random starting permutation. 

Recall the operators defined in \eqref{eq:tau}. Given a word $\mathbf{u}=(i_1,\ldots,i_r)$, let $\tau_{\mathbf{u}}=\tau_{i_r}\circ\cdots\circ\tau_{i_1}$. The operator $\tau_{(1,2,\ldots,n-1)}$ is the classical \dfn{bubble-sort map}. 

Suppose we fix $\alpha\in[0,1]$ and consider the permutation $\tau_{(1,2,\ldots,n-1)}^{\left\lfloor\alpha n\right\rfloor}(u)$, where $u$ is a uniformly random permutation in $\SSS_n$ and $n$ is large. As one can see from \cref{fig:BubbleSort}, the plot of this permutation seems to have a definitive shape bounded below by a curve defined by some function ${f_\alpha\colon[0,1]\to[0,1]}$. DiFranco \cite{DiFranco} found that \[f_\alpha(x)=\begin{cases}
    \frac{\alpha}{x+\alpha} & \text{if } x\leq 1-\alpha; \\
    1-x & \text{if } x\geq 1-\alpha.\end{cases}\]  
However, DiFranco's argument does not describe the density of the points above the bounding curve. We call the associated permutons the \dfn{standard bubble-sort permutons}. We will give a new proof of DiFranco's result that determines this density precisely. In fact, the following theorem is far more general because it deals with operators of the form $\tau_{\mathbf{w}(\sS)}$, where $\sS$ is an arbitrary order-convex shape. 

\begin{figure}[]
  \begin{center}
\includegraphics[height=5cm]{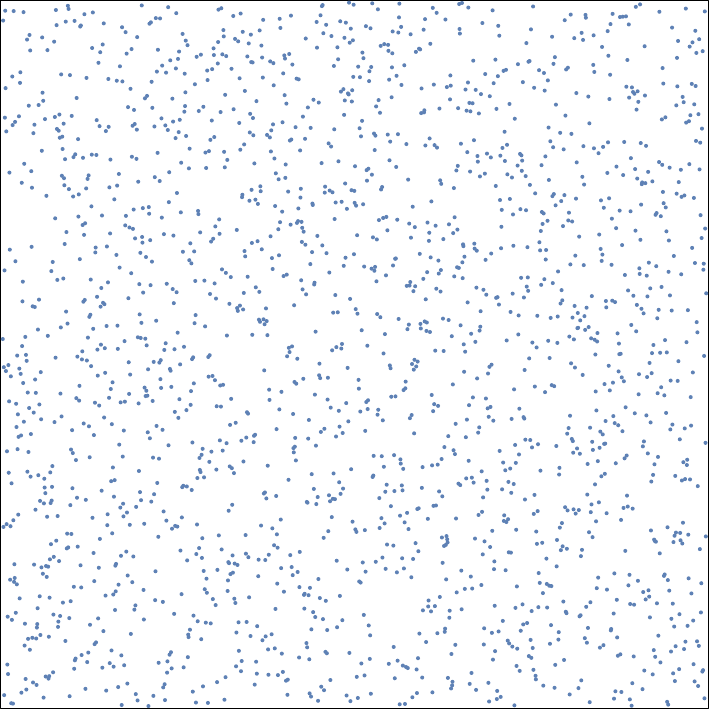}\quad\includegraphics[height=5cm]{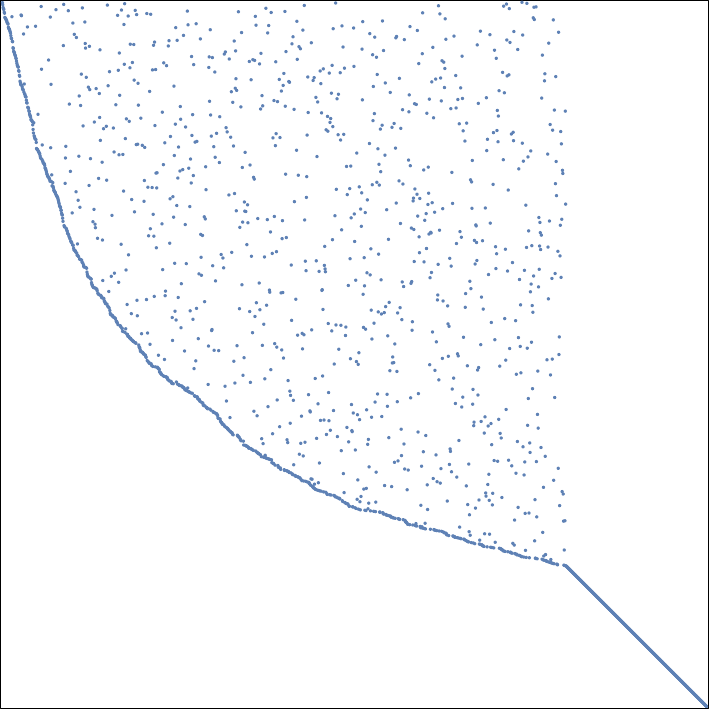}\quad\includegraphics[height=5cm]{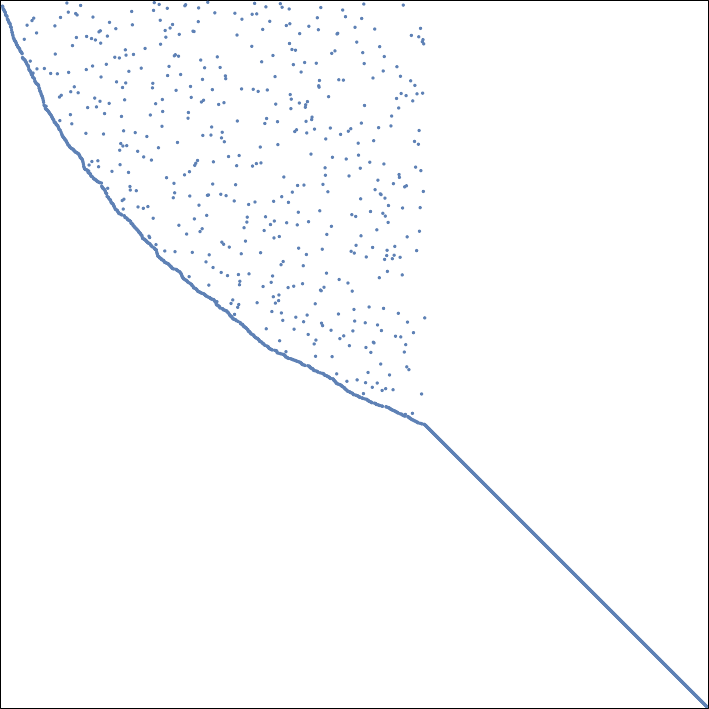} \\ \vspace{0.3cm} 
    \includegraphics[height=5cm]{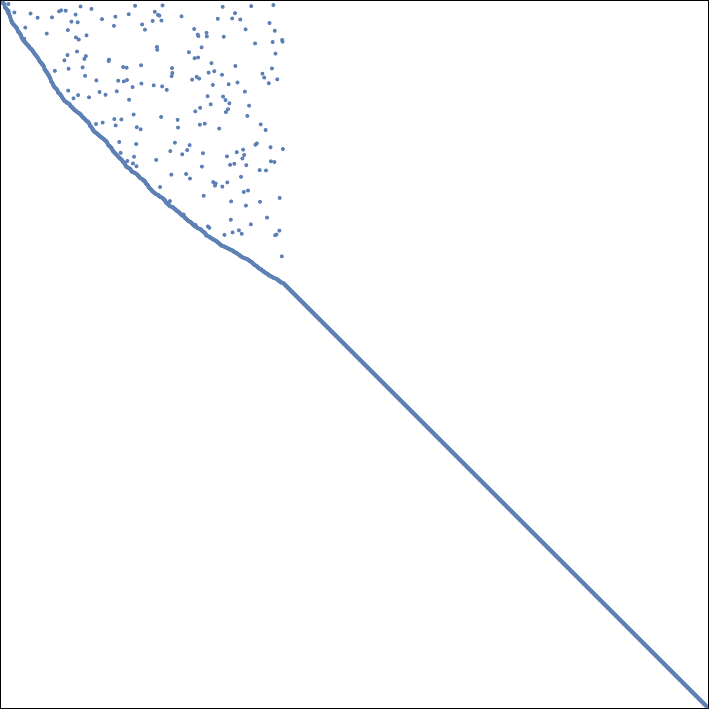}\quad\includegraphics[height=5cm]{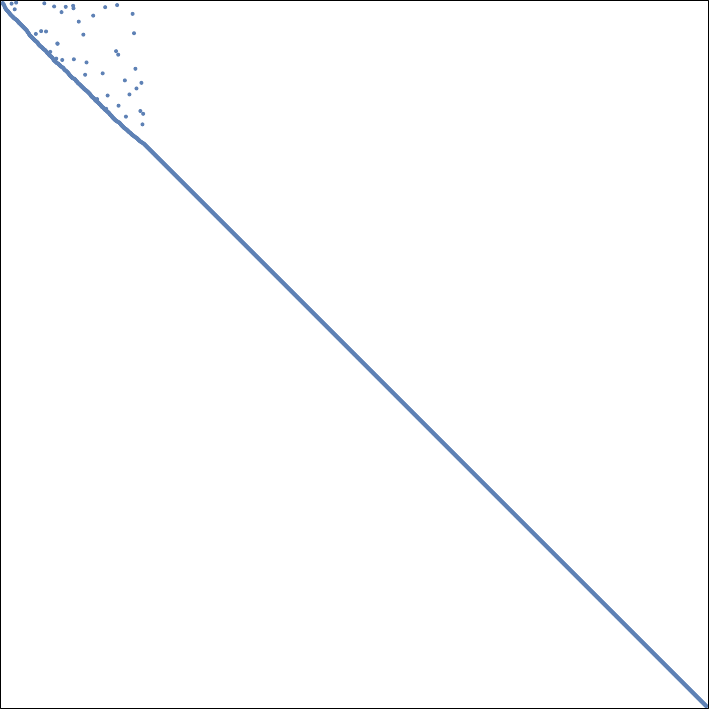}\quad\includegraphics[height=5cm]{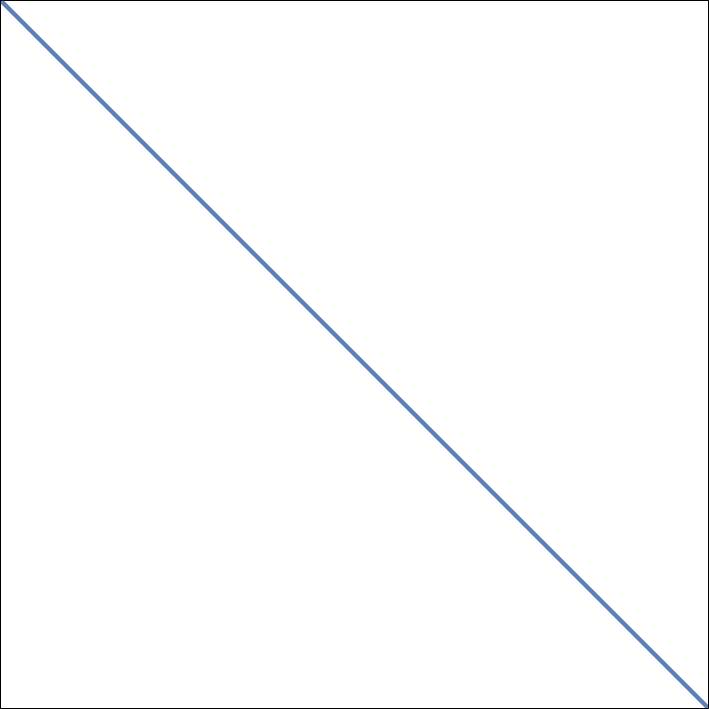} 
  \end{center}
\caption{The plots of $\tau_{(1,2,\ldots,1999)}^{400k}(u)$ for $k=0,1,2,3,4,5$, where $u$ is a uniformly random permutation in $\SSS_{2000}$.}\label{fig:BubbleSort} 
\end{figure}

\begin{theorem}\label{thm:BubbleGeneral}
Fix $\varphi,\psi\in\mathfrak B$ such that $\varphi(z)\leq\psi(z)$ for all $z\in[0,1]$. For each $n\geq 1$, choose lattice paths $\Lambda_{\swarrow}^{(n)}$ and $\Lambda_{\nearrow}^{(n)}$ in $\boldsymbol{\Lambda}_n$ such that the former lies weakly below the latter. Assume that $\theta_{\Lambda_{\swarrow}^{(n)}}$ and $\theta_{\Lambda_{\nearrow}^{(n)}}$ converge pointwise to $\varphi$ and $\psi$, respectively. Let $v_n\in\SSS_n$, and assume $(\pi_{v_n})_{n\geq 1}$ converges weakly to a permuton $\mu$. Let $u_n=\tau_{\bw(\sS^{(n)})}(v_n)$, where $\sS^{(n)}=\sS(\Lambda_{\swarrow}^{(n)},\Lambda_{\nearrow}^{(n)})$. For each ${(\xx,\yy)\in[0,1]^2}$, we have \[\lim_{n\to\infty}\h_{u_n}(\xx,\yy)=\min_{0\leq\gamma\leq 1}\left(\h_\mu(\gamma,y)+\hh_1^{\varphi,\psi}(\xx,\gamma)\right)\] with probability $1$. In particular, if $(v_n)_{n\geq 1}$ is a sequence of independent permutations such that $v_n$ is chosen uniformly at random from $\SSS_n$, then \[\lim_{n\to\infty}\h_{u_n}(\xx,\yy)=\min_{0\leq\gamma\leq 1}\left((1-\gamma)y+\hh_1^{\varphi,\psi}(\xx,\gamma)\right)\] with probability $1$.
\end{theorem} 

To illustrate \cref{thm:BubbleGeneral}, let $\mathbf{c}_n$ be a word over $[n-1]$ that uses each letter exactly once. Such a word is called a \dfn{Coxeter word}. One can think of $\tau_{\mathbf{c}_n}$ as a ``permuted bubble-sort map'' because it applies the operators $\tau_1,\ldots,\tau_{n-1}$ in a permuted order. As illustrated in \cref{fig:Coxeter_word}, the word $\mathbf{c}_n$ can be represented (modulo commutation equivalence) by a lattice path $\Lambda_{\swarrow}^{(n)}\in\boldsymbol{\Lambda}_n$ that starts at $(0,0)$, begins with a south step, and ends with an east step. Let $\Lambda_{\nearrow}^{(n)}$ be the lattice path obtained from $\Lambda_{\swarrow}^{(n)}$ by changing the first step to an east step, changing the last step to a south step, and then translating $\left\lfloor\alpha n\right\rfloor-1$ steps north and $\left\lfloor\alpha n\right\rfloor-1$ steps east. Fix a real number $\alpha>0$, and consider the operator $\tau_{\mathbf{c}_n}^{\left\lfloor\alpha n\right\rfloor}$. This operator is equal to $\tau_{\bw(\sS^{(n)})}$, where $\sS^{(n)}$ is the order-convex shape whose southwest boundary is $\Lambda_{\swarrow}^{(n)}$ and whose northeast boundary is $\Lambda_{\nearrow}^{(n)}$. Suppose the functions $\theta_{\Lambda_{\swarrow}^{(n)}}$ converge pointwise to some function $\varphi\in\mathfrak B$ as $n\to\infty$. Then the functions $\theta_{\Lambda_{\nearrow}^{(n)}}$ converge pointwise to the function $\psi\in\mathfrak B$ defined by $\psi(z)=\varphi(z)+\alpha$. Let $v_n$ be a uniformly random permutation in $\SSS_n$. In this case, \cref{thm:BubbleGeneral} guarantees that the permutons associated to the permutations $\tau_{\mathbf{c}_n}^{\left\lfloor\alpha n\right\rfloor}(v_n)$ converge to a limiting permuton that can, in theory, be computed from $\varphi$. In practice, explicitly computing this limiting permuton can be cumbersome, so we will only do so in some special cases. 

\begin{figure}[]
  \begin{center}
  \includegraphics[height=5.714cm]{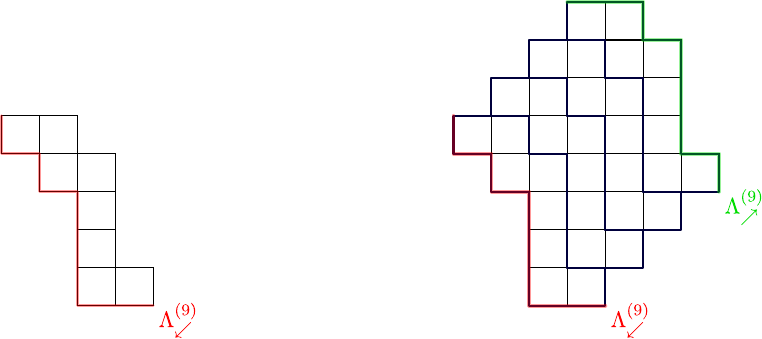}
  \end{center}
\caption{On the left is a shape representing the commutation class of the Coxeter word ${\bf c}_9=(1,3,2,7,6,5,4,8)$ for $\SSS_9$. This shape is uniquely determined by the lattice path $\Lambda_{\swarrow}^{(9)}$. On the right is the shape $\sS^{(9)}$ such that $\tau_{{\bf w}(\sS^{(9)})}=\tau_{{\bf c}_9}^4$.}\label{fig:Coxeter_word} 
\end{figure}

If our Coxeter word $\mathbf{c}_n$ is $(1,2,\ldots,n-1)$, then $\tau_{\mathbf{c}_n}$ is the classical bubble-sort map, and the limiting function $\varphi$ is given by $\varphi(z)=z$ for all $z\in[0,1]$. Another natural choice for $\mathbf{c}_n$ is the \dfn{bipartite Coxeter word} $(1,3,5,\ldots,2\left\lfloor n/2\right\rfloor-1,2,4,6,\ldots,2\left\lceil n/2\right\rceil-2)$, which lists the odd elements of $[n-1]$ before the even ones. With this choice, the limiting function $\varphi$ is given by $\varphi(z)=z/2$. 

The next theorem provides a one-parameter family of examples that includes the two mentioned in the preceding paragraph. To state it, consider $\alpha>0$ and $\beta\in[0,1]$; we will define a particular permuton $\nu^{\alpha,\beta}$. First, define curves 
\[
\mathscr{C}_{\searrow}^{\alpha,\beta}=\left\{\left(x,1-\frac{\alpha}{(1-\beta)(1-x)+\alpha}\right):\max\left\{\frac{\alpha}{1-\beta},1-\frac{\alpha}{\beta}\right\}\leq x\leq 1\right\}
\] and 
\[
\mathscr{C}_{\nwarrow}^{\alpha,\beta}=\left\{\left(x,\frac{\alpha}{\beta x+\alpha}\right):0\leq x\leq \min\left\{\frac{\alpha}{1-\beta},1-\frac{\alpha}{\beta}\right\}\right\}.
\] 
Next, define 
\begin{equation}\label{eq:UU1}
\mathscr{U}_{\searrow}^{\alpha,\beta}=\left\{\left(x,y\right)\in\left[\frac{\alpha}{1-\beta},1\right]\times\left[0,\min\left\{1-\beta,1-\frac{\alpha}{1-\beta}\right\}\right]:y\leq 1-\frac{\alpha}{(1-\beta)(1-x)+\alpha}\right\}
\end{equation} and 
\begin{equation}\label{eq:UU2}
\mathscr{U}_{\nearrow}^{\alpha,\beta}=\left\{\left(x,y\right)\in\left[0,1-\frac{\alpha}{\beta}\right]\times\left[\max\left\{1-\beta,\frac{\alpha}{\beta}\right\},1\right]:y\geq \frac{\alpha}{\beta x+\alpha}\right\}.
\end{equation} 
Consider also the (possibly empty) line segment 
\begin{equation}\label{eq:LL}
\mathscr{L}^{\alpha,\beta}=\left\{(x,1-x)\in[0,1]^2:1-\frac{\alpha}{\beta}\leq x\leq\frac{\alpha}{1-\beta}\right\}.
\end{equation} (In all of these definitions, we interpret $\alpha/0$ as $\infty$.) See the top of \cref{fig:c-Bubble} for illustrations of these curves and regions. Finally, define the permuton $\nu^{\alpha,\beta}$ so that 
\begin{align*}
\nu^{\alpha,\beta}(A)\!=\!\iint_{A\cap\left(\mathscr{U}_{\searrow}^{\alpha,\beta}\cup\mathscr{U}_{\nwarrow}^{\alpha,\beta}\right)}\!\mathrm{d}x\,\mathrm{d}y +\!\int_{A\cap\mathscr{C}_{\searrow}^{\alpha,\beta}}\!\left(x-\frac{\alpha}{1-\beta}\right)\!\mathrm{d}x+\!\int_{A\cap\mathscr{C}_{\nwarrow}^{\alpha,\beta}}\!\left(1-\frac{\alpha}{\beta}-x\right)\!\mathrm{d}x+\!\int_{A\cap\mathscr{L}^{\alpha,\beta}}\!\mathrm{d}x
\end{align*}
for every measurable set $A\subseteq[0,1]^2$. 

\begin{figure}[]
   \begin{center}
  \includegraphics[height=5cm]{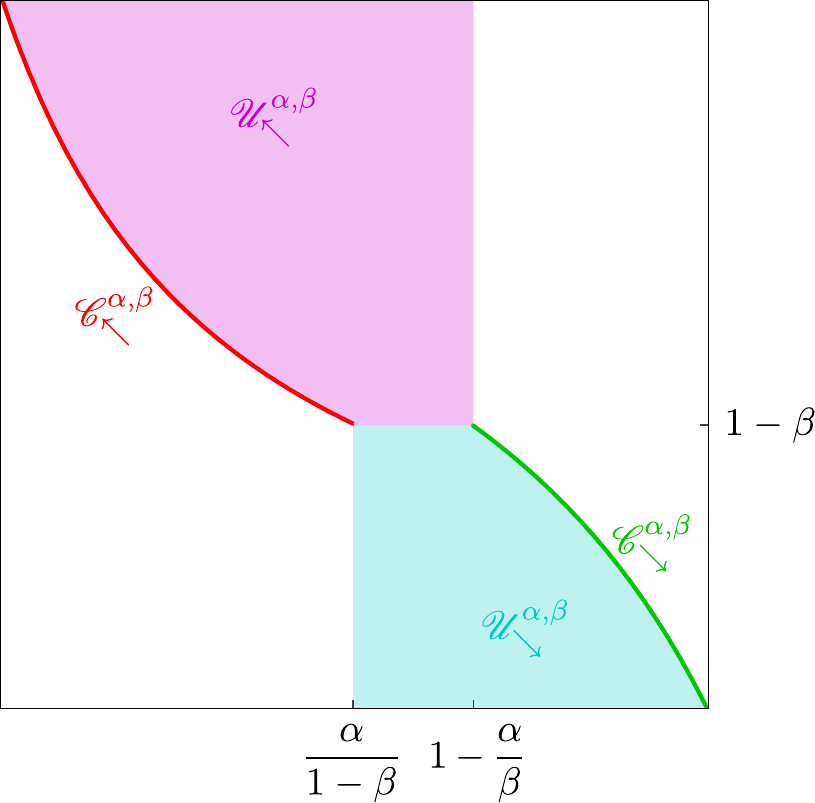}\quad\includegraphics[height=5cm]{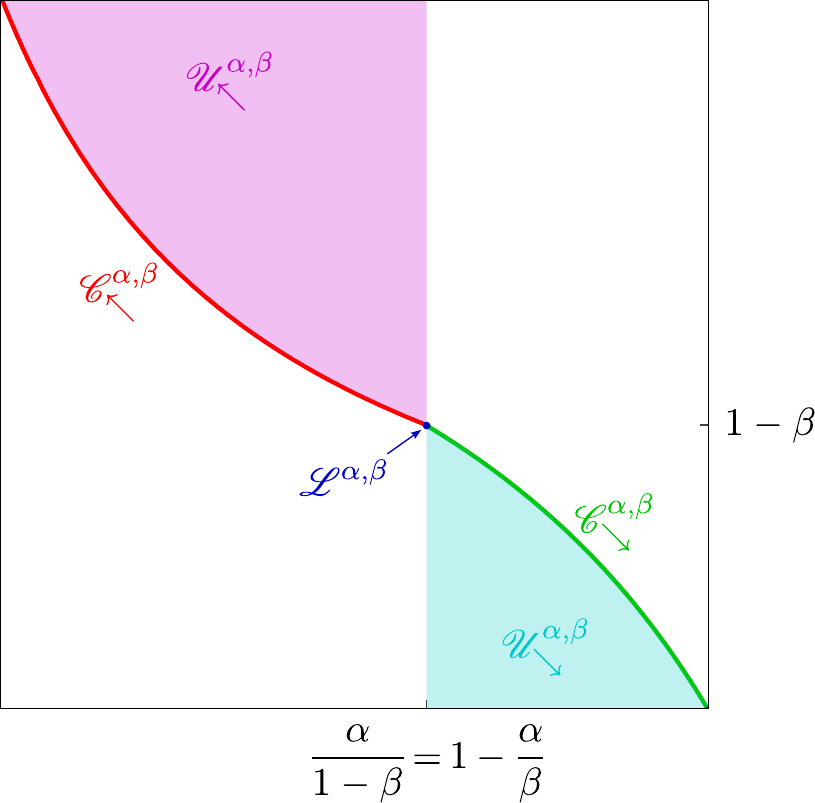}\quad\includegraphics[height=5cm]{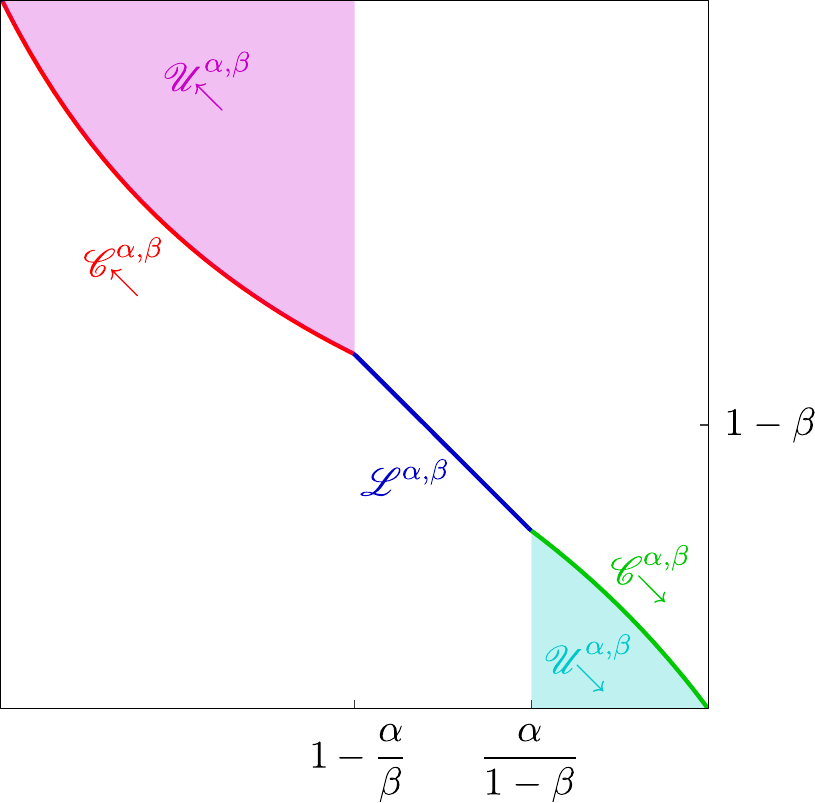} \\ \vspace{0.3cm} 
    \includegraphics[height=5cm]{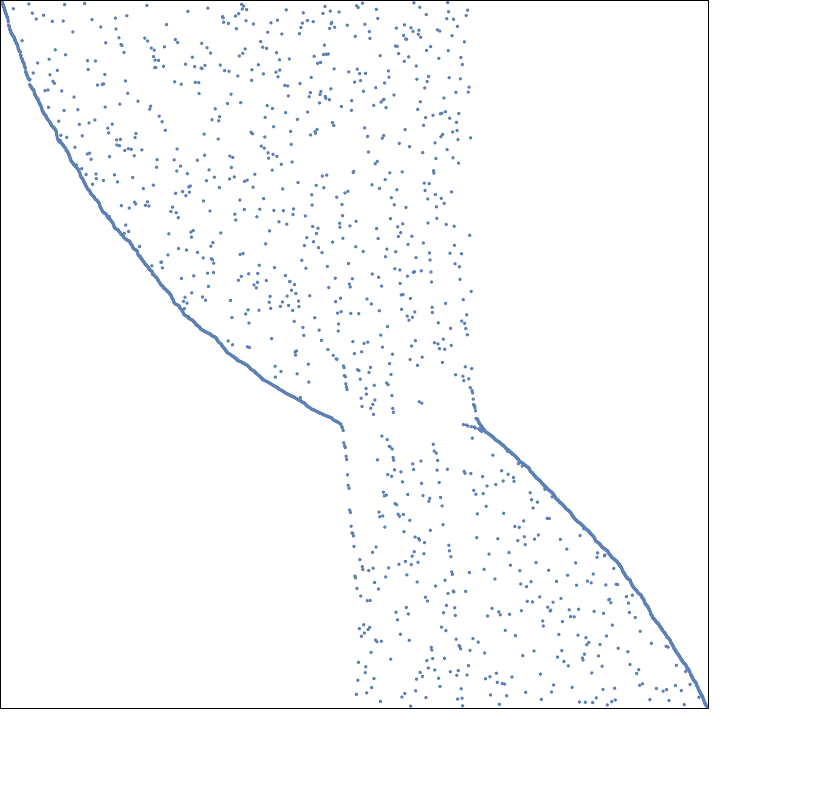}\quad\includegraphics[height=5cm]{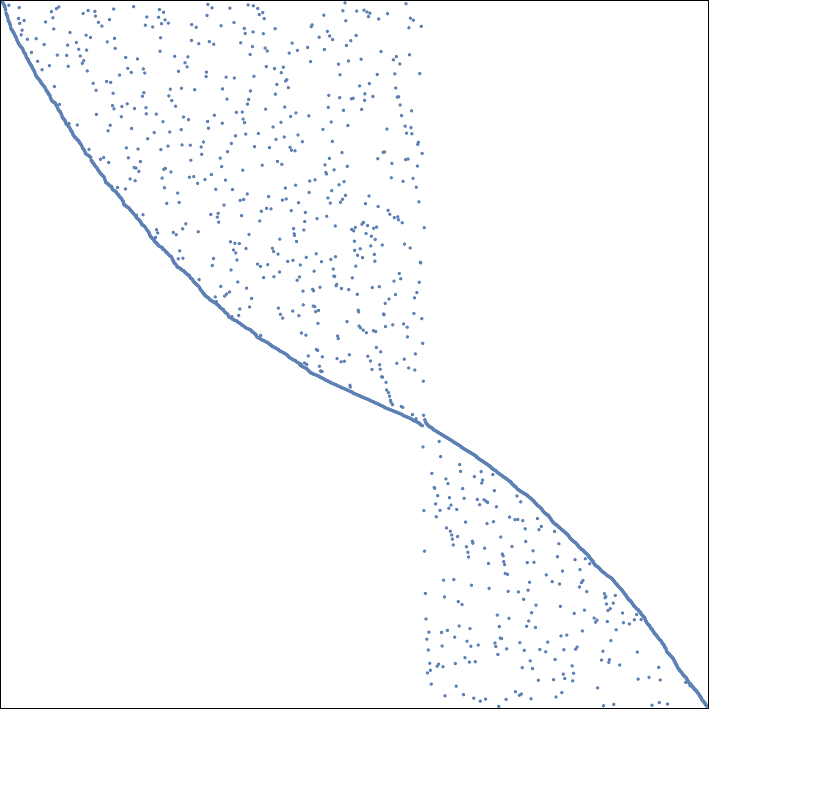}\quad\includegraphics[height=5cm]{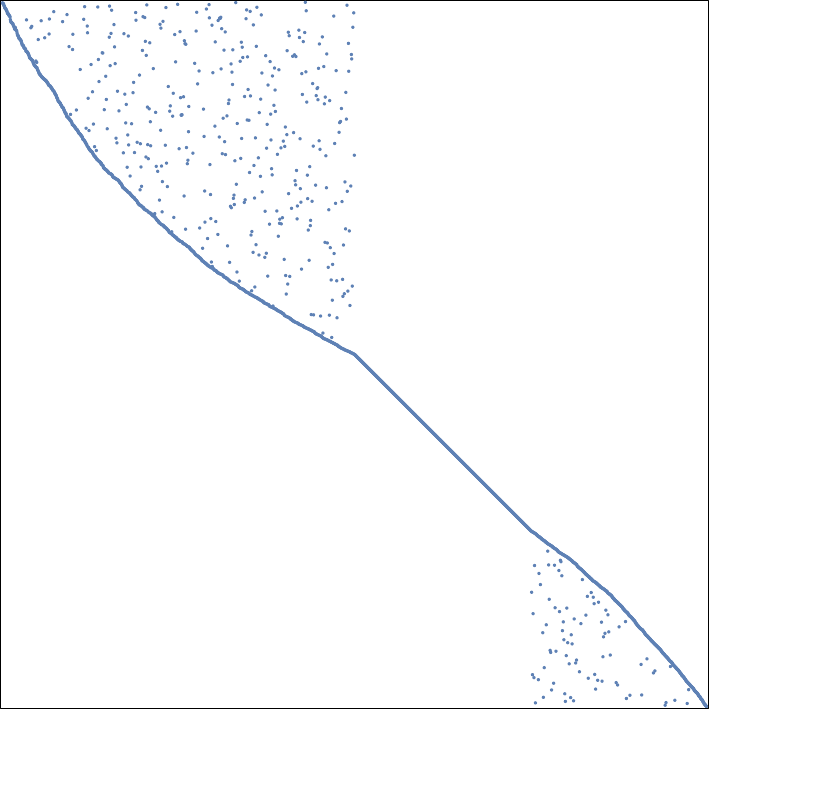} 
  \end{center}
\caption{On top are illustrations of the permutons $\nu^{\alpha,\beta}$ with $\beta=0.6$ and with $\alpha=0.2$ (left), $\alpha=0.24$ (middle), and $\alpha=0.3$ (right). Directly below each permuton is a plot of a random permutation $u_{2400}\in\SSS_{2400}$, as defined in \cref{thm:c-Bubble_Line} with the same values of $\alpha$ and $\beta$. }\label{fig:c-Bubble} 
\end{figure}

\begin{theorem}\label{thm:c-Bubble_Line} 
Fix real numbers $\alpha>0$ and $\beta\in[0,1]$. For each $n\geq 1$, let ${\bf c}_n$ be a Coxeter word over $[n-1]$ corresponding to a lattice path $\Lambda_{\swarrow}^{(n)}\in\boldsymbol{\Lambda}_n$. Suppose $\theta_{\Lambda_{\swarrow}^{(n)}}$ converges pointwise to the function $\varphi\in\mathfrak B$ given by $\varphi(z)=\beta z$. Let $v_n$ be a permutation chosen uniformly at random from $\SSS_n$, and let $u_n=\tau_{{\bf c}_n}^{\left\lfloor\alpha n\right\rfloor}(v_n)$. The sequence $(\pi_{u_n})_{n\geq 1}$ of permutons converges weakly to $\nu^{\alpha,\beta}$.  
\end{theorem} 

\cref{thm:c-Bubble_Line} is illustrated in \cref{fig:BubbleSort} for $\beta=1$ and in \cref{fig:c-Bubble} for $\beta=0.6$. 

For our next concrete application of \cref{thm:BubbleGeneral}, we consider the case where the scaling limit region $\DD$ is a rectangle. 

\begin{theorem}\label{thm:rectangle} 
Fix $\beta\in[0,1]$, and let $\sS^{(n)}$ be a rectangle shape of height $\left\lfloor(1-\beta)n\right\rfloor$ and width $\left\lceil\beta n\right\rceil$. Let $v_n$ be a permutation chosen uniformly at random from $\SSS_n$, and let $u_n=\tau_{{\bf w}(\sS^{(n)})}(v_n)$. The sequence $(\pi_{u_n})_{n\geq 1}$ of permutons converges weakly to $\nu^{\beta(1-\beta),\beta}$. 
\end{theorem} 

\begin{remark}\label{rem:doppelgangers}  
It is surprising that the permuton limit in \cref{thm:c-Bubble_Line} is the same as the permuton limit in \cref{thm:rectangle} when $\alpha=\beta(1-\beta)$. When $\sS^{(n)}$ is as in \cref{thm:rectangle}, the limit of the scaled shapes $\frac{1}{n}\sS^{(n)}$ is an axis-parallel rectangle of height $1-\beta$ and width $\beta$. On the other hand, the scaling limit of the shapes involved in \cref{thm:c-Bubble_Line} is a parallelogram with two sides of slope $1$ and length $\beta(1-\beta)\sqrt{2}$ and two other sides of slope $-\frac{1-\beta}{\beta}$. It is not clear why there should be any relationship between this rectangle and this parallelogram.  
\end{remark} 


\cref{thm:rectangle} has the following interpretation as a model of memory (in \cite{AES}, Alon, Elboim, and Sly studied a related model). Fix $\beta\in(0,1)$, and let $n$ be a large positive integer. Let $k=\left\lfloor (1-\beta) n\right \rfloor$. Dory is capable of storing exactly $k$ facts in her mind. Each fact has some \dfn{relevance factor}, which is a number in $[0,1]$ that indicates how relevant it is to Dory's life. On Day $0$, Dory knows $k$ facts, whose relevance factors are independent uniformly random elements of $[0,1]$. On each subsequent day, Dory learns a new fact whose relevance factor is a uniformly random element of $[0,1]$ that is independent of all other facts she has ever known; she then forgets the least relevant fact she knows. Let $r_j$ denote the relevance factor of the fact that Dory forgets on Day $j$. The collection of points $(1-(i-1)/n,r_i)$ for $1\leq i\leq n-k$ closely approximates the plot of the random permutation $u_n$ from \cref{thm:rectangle}, restricted to $[1-\beta,1]\times [0,1]$. We will make this more precise in \cref{subsec:diary}.  

\subsection{Demazure Products of Permutons} 

The method that we use to prove \cref{thm:BubbleGeneral} relies on an alternative formulation of the Demazure product via matrix multiplication in the min-plus tropical semiring. This formulation, due to Chan and Pflueger \cite{CP}, naturally leads us to introduce the Demazure product $\mu\star\nu$ of two permutons $\mu$ and $\nu$, which is defined so that 
\begin{equation}\label{eq:intro_h}
\h_{\mu\star\nu}(x,y)=\min_{0\leq \gamma\leq 1}\left(\h_{\mu}(\gamma,y)+\h_{\nu}(x,\gamma)\right)
\end{equation} for all $(x,y)\in[0,1]^2$. Suppose $(u_n)_{n\geq 1}$ and $(v_n)_{n\geq 1}$ are sequences of permutations with $u_n,v_n\in\SSS_n$ for all $n$. In \cref{prop:Demazure_permuton}, we prove that if $\pi_{u_n}$ and $\pi_{v_n}$ converge to permutons $\mu$ and $\nu$, respectively, then $\pi_{u_n\star v_n}$ converges to $\mu\star\nu$. 

The Demazure product on permutons allows us to give a simple proof of the following theorem, which is not obvious if we use the original definition of the Demazure product.  

\begin{theorem}\label{thm:random_star_random}
For each $n\geq 1$, choose permutations $u_n,v_n\in\SSS_n$ independently and uniformly at random. The expected number of inversions of $u_n\star v_n$ satisfies  
\[\mathbb E[\ell(u_n\star v_n)]=\binom{n}{2}(1-o(1)).\]
\end{theorem} 

Suppose $u$ and $v$ are independent uniformly random permutations in $\SSS_n$, and let $\mathbf{u}$ and $\mathbf{v}$ be reduced words for $u$ and $v$, respectively. To compute $u\star v$ using words, we use the relations of the $0$-Hecke monoid to transform the concatenated word $\mathbf{u}\mathbf{v}$ into a reduced word. The length of $\mathbf{u}\mathbf{v}$ is $\ell(u)+\ell(v)$, which is concentrated around $\binom{n}{2}$. Therefore, \cref{thm:random_star_random} tells us that the process of transforming $\mathbf{u}\mathbf{v}$ into a reduced word does not shorten the word by much. That is to say, $\mathbf{u}\mathbf{v}$ is already fairly close to a reduced word. It would be interesting to have more refined asymptotics for the expected value (or other distributional information) of $\ell(u\star v)$.   

The Demazure product on permutons could have further applications. Let us briefly mention one potential avenue that we will not pursue here. Fix $\alpha>0$, and let \[z_n=(\tau_{i_{\left\lfloor\alpha n\right\rfloor}}\circ\cdots\circ\tau_{i_2}\circ\tau_{i_1})(\mathrm{id}_n)=s_{i_1}\star s_{i_2}\star \cdots\star s_{i_{\left\lfloor\alpha n\right\rfloor}},\] where $i_1,i_2,\ldots,i_{\left\lfloor\alpha n\right\rfloor}$ are independent uniformly random elements of $[n-1]$ and $\mathrm{id}_n$ is the identity permutation in $\SSS_n$. The permutation $z_n$ is essentially\footnote{The oriented swap process is typically defined using continuous time, whereas we work in discrete time.} the same as a state at time $\left\lfloor\alpha n\right\rfloor$ in the \dfn{oriented swap process}, which was introduced by Angel, Holroyd, and Romik \cite{AHR} (see also \cite{BCGR,BGR,Zhang}). These authors found that $(\pi_{z_n})_{n\geq 1}$ almost surely converges weakly to a particular permuton $\lambda^\alpha$, and they provided an explicit formula for the height function $\h_{\lambda^\alpha}$. Now suppose we instead consider the oriented swap process with an initial state $v_n$ other than the identity permutation. That is, we consider $(\tau_{i_{\left\lfloor\alpha n\right\rfloor}}\circ\cdots\tau_{i_2}\circ\tau_{i_1})(v_n)=v_n\star z_n$. If we further assume that $(\pi_{v_n})_{n\geq 1}$ almost surely converges to some permuton $\mu$ (for instance, we could take $\mu$ to be the uniform measure on $[0,1]^2$), then it will follow from \cref{prop:Demazure_permuton} that $(\pi_{v_n\star z_n})_{n\geq 1}$ almost surely converges weakly to the permuton $\mu\star\lambda^{\alpha}$, whose height function can be computed directly from those of $\mu$ and $\lambda^{\alpha}$ via \eqref{eq:intro_h}. 

\subsection{Outline} 
In \cref{sec:preliminaries}, we provide additional background information about permutons. 
In \cref{sec:Demazure_permutons}, we define the Demazure product on the set of permutons and establish its basic properties (\cref{prop:Demazure_permuton}). This allows us to painlessly prove \cref{thm:random_star_random}; in fact, we will deduce this theorem from a more general result about pattern densities in the Demazure product of two large uniformly random permutations (\cref{thm:random_star_random_patterns}). \cref{sec:TASEP} discusses the TASEP with geometric jumps and proves \cref{thm:NewMPPY,thm:fluctuations}, which deal with permutons associated to Demazure products coming from random pipe dreams. \cref{sec:TASEP} also provides more thorough analyses of families of permutons arising in special cases. In \cref{sec:Bubble}, we prove \cref{thm:BubbleGeneral} and then derive \cref{thm:c-Bubble_Line,thm:rectangle}. \cref{sec:conclusion} provides ideas for future directions along with more pictures of large random permutations. 

\section{Preliminaries}\label{sec:preliminaries}  

Recall that $\pi_u$ is the permuton obtained from the plot of a permutation $u\in\SSS_n$ by changing each point into a square of side length $1/n$ and density $n$. We write $\h_u$ instead of $\h_{\pi_u}$ for the height function of $\pi_u$. It is straightforward to check that the height function of a permuton $\mu$, as defined in \eqref{eq:height}, is $\sqrt{2}$-Lipschitz and, therefore, continuous. A sequence $(\mu_n)_{n\geq 1}$ of permutons converges weakly to a permuton $\mu$ if and only if the sequence $(\h_{\mu_n})_{n\geq 1}$ converges pointwise to $\h_\mu$. Moreover, since height functions are $\sqrt{2}$-Lipschitz, the sequence $(\h_{\mu_n})_{n\geq 1}$ converges pointwise if and only if it converges uniformly.  

\begin{proposition}[\cite{HKMRS}]\label{lem:permuton_convergence}
If $\mu$ is a permuton, then there exists a sequence $(u_n)_{n\geq 1}$ with $u_n\in\SSS_n$ such that the sequence $(\h_{u_n})_{n\geq 1}$ converges pointwise to $\h_\mu$. Moreover, if $(v_n)_{n\geq 1}$ is a sequence of permutations with $v_n\in\SSS_n$ such that $(v_n)_{n\geq 1}$ converges pointwise to a function $\h\colon[0,1]^2\to[0,1]$, then $\h$ is the height function of a permuton. 
\end{proposition} 

A permuton $\mu$ can be used to generate random permutations as follows. First, given a sequence $r_1,\ldots,r_n$ of $n$ distinct real numbers, there is a unique permutation $u\in\SSS_n$ such that for all $j,j'\in[n]$, we have $r_j<r_{j'}$ if and only if $u(j)<u(j')$; we say the sequence $r_1,\ldots,r_n$ has the \dfn{same relative order} as $u$. Let us choose $n$ points in $[0,1]^2$ independently at random according to $\mu$. With probability $1$, no two of these points will lie on the same vertical or horizontal line. Thus, we can order these points as $(x_1,y_1),\ldots,(x_n,y_n)$, where $x_1<\cdots<x_n$. The sequence $y_1,\ldots,y_n$ has the same relative order as a unique permutation $w$. We call $w$ a \dfn{$\mu$-random permutation} of size $n$. 

One of the original motivations for the introduction of permutons in \cite{HKMRS} came from the study of \emph{pattern densities}. Let $k\leq n$ be positive integers, and let $u\in\SSS_n$ and $v\in\SSS_k$. We say a tuple $(i_1,\ldots,i_k)$ of integers with $1\leq i_1<\cdots<i_k\leq n$ is a \dfn{pattern occurrence} of $v$ in $u$ if the sequence $u(i_1),\ldots,u(i_k)$ has the same relative order as $v$. Let $\mathcal{N}_v(u)$ be the number of pattern occurrences of $v$ in $u$, and let $\mathcal{D}_v(u)=\frac{1}{\binom{n}{k}}\mathcal{N}_v(u)$. For example, $\mathcal{N}_{21}(u)$ is equal to $\ell(u)$, the number of inversions of $u$. Let us also write $\mathcal{D}_v(\mu)$ for the probability that a $\mu$-random permutation of size $k$ equals~$v$. 

\begin{proposition}[\cite{HKMRS}]\label{prop:patterns} 
Let $(u_n)_{n\geq 1}$ be a sequence of permutations with $u_n\in\SSS_n$ for all $n$. Let $\mu$ be a permuton. Then $(\pi_{u_n})_{n\geq 1}$ converges weakly to $\mu$ if and only if 
\[\lim_{n\to\infty}\mathcal{D}_v(u_n)=\mathcal D_v(\mu)\] for every permutation $v$. 
\end{proposition} 

Given a statement $\mathrm{P}$, we let 
\[\boldsymbol{1}_{\mathrm{P}}=\begin{cases}
    1 & \text{if P is true}; \\
    0 & \text{if P is false}. 
\end{cases}\]

\section{Demazure Products of Permutons}\label{sec:Demazure_permutons} 

A sketch of a proof of the following result was given by Chan and Pflueger in \cite{CP}; a full proof was later provided by Pflueger in \cite{Pflueger}. The result was phrased in the language of permutations rather than permutons, but the translation from one to the other is straightforward. 

\begin{theorem}[\cite{CP,Pflueger}]\label{thm:CP} 
Let $u,v\in\SSS_n$. For all $x,y\in[0,1]$, we have 
\[\h_{u\star v}(x,y)=\min_{0\leq \gamma\leq 1}(\h_u(\gamma,y)+\h_v(x,\gamma)).\]
\end{theorem}

The preceding theorem naturally leads us to the notion of a Demazure product for permutons. 

\begin{proposition}\label{prop:Demazure_permuton}
Let $\mu,\nu$ be permutons. There is a unique permuton $\mu\star\nu$ such that 
\[\h_{\mu\star\nu}(x,y)=\min_{0\leq\gamma\leq 1}(\h_\mu(\gamma,y)+\h_\nu(x,\gamma))\] for all $x,y\in[0,1]$. Moreover, if $(u_n)_{n\geq 1}$ and $(v_n)_{n\geq 1}$ are sequences with $u_n,v_n\in\SSS_n$ such that $(\pi_{u_n})_{n\geq 1}$ converges weakly to $\mu$ and $(\pi_{v_n})_{n\geq 1}$ converges weakly to $\nu$, then the sequence $(\pi_{u_n\star v_n})_{n\geq 1}$ converges weakly to $\mu\star\nu$. 
\end{proposition} 

\begin{proof}
For $x,y\in[0,1]$, let \[\h(x,y)=\min_{0\leq\gamma\leq 1}(\h_\mu(\gamma,y)+\h_\nu(x,\gamma)).\] By \cref{lem:permuton_convergence}, there exist sequences $(u_n)_{n\geq 1}$ and $(v_n)_{n\geq 1}$ such that $(\h_{u_n})_{n\geq 1}$ converges pointwise to $\h_\mu$ and $(\h_{v_n})_{n\geq 1}$ converges pointwise to $\h_\nu$. We will prove that $(\h_{u_n\star v_n})_{n\geq 1}$ converges pointwise to $\h$. Since weak convergence of a sequence of permutons is equivalent to pointwise convergence of the corresponding sequence of height functions, this will complete the proof. 

A sequence of height functions converges pointwise if and only if it converges uniformly (because height functions are $\sqrt{2}$-Lipschitz). Thus, $(\h_{u_n})_{n\geq 1}$ and $(\h_{v_n})_{n\geq 1}$ converge uniformly to $\h_\mu$ and $\h_\nu$, respectively. Choose $\epsilon>0$. There is a positive integer $N$ such that \[|\h_{u_n}(x,y)-\h_\mu(x,y)|<\epsilon\quad\text{and}\quad|\h_{v_n}(x,y)-\h_\nu(x,y)|<\epsilon\] for all $(x,y)\in[0,1]^2$ and all integers $n\geq N$. For $(x,y)\in[0,1]^2$ and $n\geq N$, we can use \cref{thm:CP} to find that 
\begin{align*}
\h_{u_n\star v_n}(x,y)&=\min_{0\leq\gamma\leq 1}\left(\h_{u_n}(\gamma,y)+\h_{v_n}(x,\gamma)\right) \\ 
&\geq\min_{0\leq\gamma\leq 1}\left(\h_{\mu}(\gamma,y)-\epsilon+\h_{\nu}(x,\gamma)-\epsilon\right) \\ 
&=\h(x,y)-2\epsilon. 
\end{align*}
A similar computation shows that $\h_{u_n\star v_n}(x,y)\leq\h(x,y)+2\epsilon$. As $\epsilon$ was arbitrary, this proves the desired convergence.  
\end{proof}

Given permutons $\mu$ and $\nu$, we of course call $\mu\star\nu$ the \dfn{Demazure product} of $\mu$ and $\nu$. Note that the Demazure product is an associative binary operation on the set of permutons. Indeed, this follows from \cref{prop:Demazure_permuton} and the fact that the usual Demazure product on $\SSS_n$ is associative. The \dfn{identity permuton} is the permuton supported on the line segment $\{(x,x):x\in[0,1]\}$; one can think of it as the scaling limit of identity permutations. The set of permutons forms a monoid under the Demazure product. 

Let us write $\delta_k$ for the decreasing permutation in $\SSS_k$. We now have the tools to effortlessly prove the following theorem, which contains \cref{thm:random_star_random} as a special case since the number of inversions of a permutation $u$ is the same as $\mathcal{N}_{\delta_2}(u)$. 

\begin{theorem}\label{thm:random_star_random_patterns} 
Fix a permutation $w$ of size $k$. Let $(u_n)_{n\geq 1}$ and $(v_n)_{n\geq 1}$ be independent sequences of independent permutations, where $u_n$ and $v_n$ are chosen uniformly at random from $\SSS_n$. We have 
\[\lim_{n\to\infty}\mathcal{D}_w(u_n\star v_n)=\boldsymbol{1}_{w=\delta_k}.\] 
\end{theorem}

\begin{proof}
Choose $u_n,v_n\in\SSS_n$ independently and uniformly at random. Each of the sequences $(\pi_{u_n})_{n\geq 1}$ and $(\pi_{v_n})_{n\geq 1}$ converges weakly to the uniform probability measure $\upsilon$ on $[0,1]^2$. By \cref{prop:Demazure_permuton}, the sequence $(\h_{u_n\star v_n})_{n\geq 1}$ converges pointwise to $\h_{\upsilon\star\upsilon}$. Note that $\h_{\upsilon}(x,y)=(1-x)y$ for all $x,y\in[0,1]$. By \cref{prop:Demazure_permuton}, we have 
\begin{align*}
\h_{\upsilon\star\upsilon}(x,y)&=\min_{0\leq \gamma\leq 1}((1-\gamma)y+(1-x)\gamma) \\ 
&=\begin{cases}
    y & \text{if } x+y\leq 1; \\
    1-x & \text{if } x+y>1. 
\end{cases}
\end{align*} 
It follows that $\upsilon\star\upsilon$ is the permuton whose support is the line segment $\{(x,1-x):x\in[0,1]\}$.  Consequently, an $(\upsilon\star\upsilon)$-random permutation of size $k$ is $\delta_k$ with probability $1$. In other words, $\mathcal{D}_w(\upsilon\star\upsilon)=\boldsymbol{1}_{w=\delta_k}$. The desired result now follows from \cref{prop:patterns}.    
\end{proof}

\section{Random Pipe Dreams and the TASEP}\label{sec:TASEP}  

Our first goal in this section is to prove \cref{thm:NewMPPY,thm:fluctuations}. To do so, we follow the approach of \cite{MPPY} by relating random pipe dreams to the totally asymmetric simple exclusion process (TASEP) with geometric jumps and parallel updates, which dates back to the work of Vershik and Kerov \cite{VK} and appears in other articles such as \cite{BF,DMO,DW,KPS}. The authors of \cite{MPPY} first related random pipe dreams to the stochastic $6$-vertex model and then appealed to the known connections between the stochastic $6$-vertex model and the TASEP. We will instead pass directly from random pipe dreams to the TASEP.  

When proving \cref{thm:NewMPPY}, we will make a simplifying assumption that $p\neq 1$. Essentially the same argument handles the case in which $p=1$, but including this case makes the notation a bit more cumbersome. In fact, the argument is much simpler when $p=1$ because there is no randomness involved. (One can think of the TASEP with geometric jumps with parameter $p=1$ as a deterministic process that moves one particle off to infinity at each time step until there are no particles left.) 

\subsection{The TASEP with Geometric Jumps} 
Fix $p\in(0,1)$. Consider the following discrete-time interacting particle system with $k$ particles occupying distinct positions on the positive integer lattice $\ZZ_{\geq 1}$. For $i\in[k]$ and $t\in\ZZ_{\geq 0}$, let $\xi_i(t)\in\ZZ_{\geq 1}$ denote the position of particle $i$ at time $t$. We use \dfn{step initial conditions}, meaning $\xi_i(0)=k+1-i$ for all $i\in[k]$ (so particle $1$ is farthest to the right). At each time step, the particles $k,k-1,\ldots,1$ jump in this order to the right by distances determined by independent geometric random variables with parameter $p$, subject to the conditions that no two particles can occupy the same site and that no particle can jump over another particle. To be more precise, let $(G_i(t))_{(i,t)\in[k]\times\ZZ_{\geq 1}}$ be a collection of independent geometric random variables, each with parameter $p$. Thus, \[\mathbb P(G_i(t)=r)=(1-p)p^r\] for all nonnegative integers $r\geq 0$. Let $\xi_0(t)=\infty$ for all $t\geq 0$. Then for all $t\geq 1$, we have
\[\xi_i(t)=\min(\xi_i(t-1)+G_i(t),\,\xi_{i-1}(t-1)-1).\] See \cref{fig:TASEP_jumps}. Let $\overline\xi_i(t)=\xi_i(t)-\xi_i(0)$.

\begin{figure}[]
  \begin{center}
  \includegraphics[height=3.35cm]{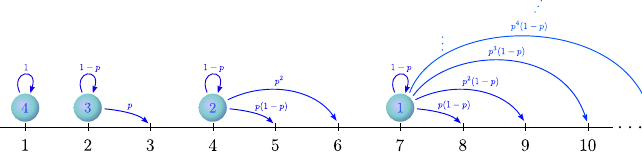}
  \end{center}
\caption{A state of the TASEP with geometric jumps and $4$ particles. In the next transition, each particle jumps to a position indicated by an arrow with the probability labeling the arrow. }\label{fig:TASEP_jumps}  
\end{figure}

The interacting particle system described above is the \dfn{TASEP with geometric jumps}. It is an integrable system that belongs to the KPZ universality class. To be more precise, we work in the hydrodynamic scale by fixing constants ${\sf m},{\sf t}\geq 0$ and letting 
\[
m=\left\lfloor L{\sf m}\right\rfloor\quad\text{and}\quad t=\left\lfloor L{\sf t}\right\rfloor, 
\]
where $L>0$ is a real parameter. The next result, which was proven in \cite{MPPY}, describes the asymptotics of quantities of the form $\mathbb P(\overline \xi_m(t)\geq L{\sf g})$ as $L\to\infty$. To state it, let us define 
\begin{equation}\label{eq:qr}
{\sf q}_p({\sf m},{\sf t})=\frac{\sqrt{p}{\sf m}^{1/3}}{{\sf t}^{1/6}}\frac{(\sqrt{{\sf t}/p}-\sqrt{\sf m})^{2/3}}{(\sqrt{p{\sf t}}-\sqrt{\sf m})^{1/3}}\quad\text{and}\quad{\sf r}_p({\sf m},{\sf t})=\frac{\sqrt{p}(\sqrt{p{\sf t}}-\sqrt{\sf m})^{2/3}(\sqrt{{\sf t}/p}-\sqrt{\sf m})^{2/3}}{({\sf mt})^{1/6}(1-p)}.
\end{equation} 
Let $F_2$ denote the cumulative distribution function of the Tracy--Widom GUE distribution. 
For $a,b\geq 0$, let 
\[
\cc_p(a,b)=\frac{(\sqrt{pb}-\sqrt{a})^2}{1-p}\boldsymbol{1}_{pb\geq a}. 
\] 

\begin{theorem}[{\cite[Theorem~3.7]{MPPY}}]\label{thm:TASEP}
We have \[\lim_{L\to\infty}L^{-1}\overline\xi_{\left\lfloor L{\sf m}\right\rfloor}(\left\lfloor L{\sf t}\right\rfloor)=\cc_p({\sf m},{\sf t}),\] where the convergence is in probability. 
Moreover, if $0<{\sf m}/p<{\sf t}$, then for all $\alpha,\beta\in\mathbb R$, we have 
\[\lim_{L\to\infty}\mathbb{P}\left(\overline{\xi}_{\left\lfloor L{\sf m}-L^{1/3}\alpha{\sf q}_p({\sf m},{\sf t})\right\rfloor}(\left\lfloor L{\sf t}\right \rfloor)\geq L\cc_p({\sf m},{\sf t})-L^{1/3}\beta{\sf r}_p({\sf m},{\sf t})\right)=F_2(\alpha+\beta).\]  
\end{theorem} 

\subsection{TASEP from Pipe Dreams} 
Let $\mathscr{S}=\mathscr{S}(\Lambda_{\swarrow},\Lambda_{\nearrow})$ be the order-convex shape bounded by two paths $\Lambda_{\swarrow}$ and $\Lambda_{\nearrow}$ in $\boldsymbol{\Lambda}_n$ and the lines $\mathfrak{L}_0$ and $\mathfrak{L}_n$. Let $r$ be the number of columns of $\sS$. Let $a_j$ and $b_j$ be the contents of the northmost and southmost boxes, respectively, in the $j$-th column of $\sS$. Let $\omega_j=(b_j,b_{j}-1,\ldots,a_j)$ be the word that lists the integers in the interval $[a_j,b_j]$ in decreasing order. A simple yet crucial observation is that 
\begin{equation}\label{eq:a_j_b_j}
a_{j}\leq\max\{a_{j+1}-1,1\}\quad\text{and}\quad b_{j+1}\geq \min\{b_j+1,n-1\}
\end{equation}
for all $1\leq j\leq r-1$. 

Let \[{u=\Delta_p\left(\sS\left(\Lambda_{\swarrow},\Lambda_{\nearrow}\right)\right)}.\] We compute $u$ by choosing a linear extension of $\sS$, writing down the associated word $\mathbf{w}(\sS)$, randomly deleting letters to form the subword $\sub_p(\mathbf{w}(\sS))$, and then taking the Demazure product. Let us choose the linear extension that traverses each column from south to north, visiting the columns one by one from west to east. Then $\mathbf{w}(\sS)=\omega_1\cdots\omega_r$. Let $v_j=\Delta_p(\omega_1\cdots\omega_r)$. 

We wish to understand $\h_u(x,y)=\h_{v_r}(x,y)$, where $y=(n-k)/n$ and $x=k'/n$ for some integers $k,k'\in[0,n]$. To this end, let $T$ be the largest integer such that  $b_T\leq n-k-1$, and let $T'$ be the largest integer such that $a_{T'}\leq k'$.  Given a permutation $w\in\SSS_n$, we obtain a state $\iota(w)$ of the $k$-particle TASEP in which the positions occupied by the particles are the elements of the set \[n+1-w^{-1}([n+1-k,n])=\{n+1-w^{-1}(i):n+1-k\leq i\leq n\}.\] A straightforward computation shows that $(x-y+\h_w(x,y))n$ is equal to the number of particles in $\iota(w)$ that occupy positions at or to the right of $n-k'+1$. 

The boxes in columns $1,\ldots, T$ have no effect on the particles in our TASEP. More precisely, since the contents of the boxes in these columns are all at most $n-k-1$ (by the definition of $T$), we have $v_{T}^{-1}([n+1-k,n])=[n+1-k,n]$, so the particles occupy their initial positions $1,\ldots,k$ in the state $\iota(v_{T})$. 

If $T+1\leq T'$, then to obtain the state $\iota(v_{T+1})$ from the state $\iota(v_{T})$, we move particle $1$ by some geometrically-distributed (with parameter $p$) amount subject to the condition that the particle cannot move past position $n+1-a_{T+1}$; all other particles stay still. Since $a_{T+1}\leq k'$ (by the definition of $T'$), it follows that the number of particles at or to the right of position $n+1-k'$ in $\iota(v_{T+1})$ has the same distribution as the maximum of $k'$ and the number of particles at or to the right of position $n+1-k'$ in the $k$-particle TASEP with geometric jumps after $1$ time step. More generally, if $T+1\leq j\leq T'$, then we obtain the state $\iota(v_{j})$ from the state $\iota(v_{j-1})$ by allowing the particles in positions at or to the left of position $n+1-a_j$ to jump by some geometrically-distributed (with parameter $p$) amounts subject to the conditions that particles cannot hop over each other and that particles cannot move past position $n+1-a_j$. Since $a_{j}\leq k'$ (by the definition of $T'$), it follows that the number of particles at or to the right of position $n+1-k'$ in $\iota(v_{j})$ has the same distribution as the maximum of $k'$ and the number of particles at or to the right of position $n+1-k'$ in the $k$-particle TASEP with geometric jumps after $j-T$ time steps. In this analysis, we are crucially using \eqref{eq:a_j_b_j} to see that the particles sitting in their initial positions that are capable of moving when we transition from $\iota(v_{j-1})$ to $\iota(v_j)$ are the same as the particles sitting in their initial positions that are capable of moving in the $(j-T)$-th step of the TASEP. When we transform $\iota(v_{T'})$ into $\iota(v_r)$ through a sequence of transitions, no particles can move strictly to the right of position $n-k'$ (by the definition of $T'$). 

The preceding paragraphs tell us that $(x-y+\h_u(x,y))n$, which is the number of particles in $\iota(u)=\iota(v_{r})$ occupying positions at or to the right of $n+1-k'$, has the same distribution as the maximum of $k'$ and the number of particles in the TASEP in these positions at time step $T'-T$. Note that $\frac{1}{n}T'\approx\theta_{\Lambda_{\nearrow}}(x)$ and $\frac{1}{n}T\approx\theta_{\Lambda_{\swarrow}}(y)$. 

\begin{figure}[]
  \begin{center}
\includegraphics[height=3.864cm]{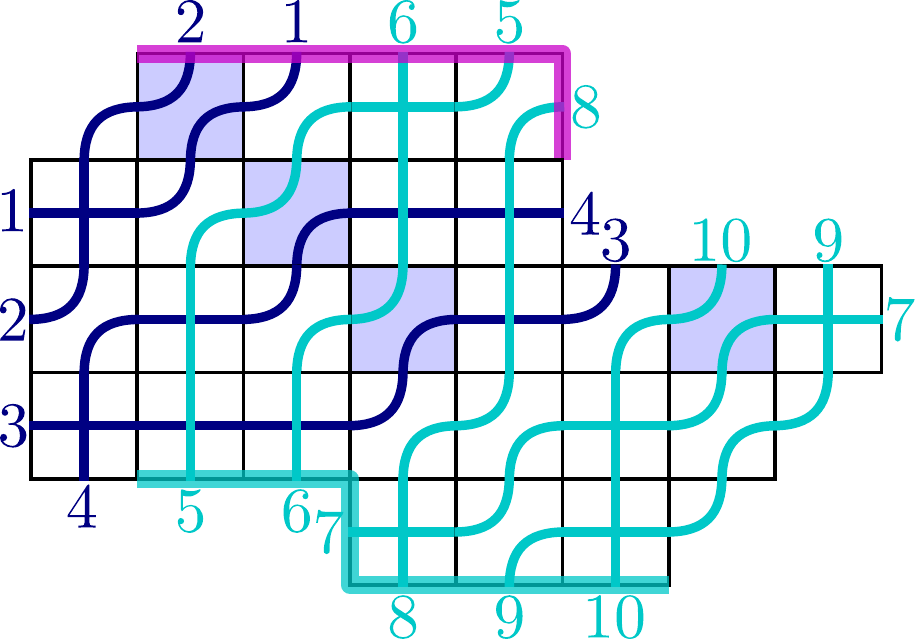}
  \end{center}
\caption{A shape filled with a random pipe dream. Pipes numbered $5,6,7,8,9,10$ are colored {\color{Teal}teal}. Shaded boxes are where crossings were resolved.  }\label{fig:TASEP_example_1}  
\end{figure} 

\begin{example}
Let $n=10$, and consider the shape shown in \cref{fig:TASEP_example_1}. We have filled this shape with a random pipe dream and resolved crossings in the shaded boxes to compute the Demazure product. We have \[(a_1,\ldots,a_8)=(1,1,2,3,4,7,8,9)\quad\text{and}\quad (b_1,\ldots,b_8)=(3,4,5,7,8,9,9,9).\]
Let $k=6$ and $k'=5$ so that $x=1/2$ and $y=2/5$. Then $T=1$ and $T'=5$. 

\cref{fig:TASEP_example_2} illustrates how to fill the shape one column at a time. At the beginning, all boxes are filled with bump tiles and colored {\color{Yellow}yellow}. At time $j$, we fill the boxes in column $j$ with the appropriate tiles from \cref{fig:TASEP_example_1} and remove the {\color{Yellow}yellow} shading from those boxes, resolving crosses and shading boxes when necessary in order to keep the pipe dream reduced. Thus, at time $j$, the columns to the right of column $j$ remain {\color{Yellow}yellow}. The pipes numbered with elements of $[n+1-k,n]=\{5,6,7,8,9,10\}$ are colored {\color{Teal}teal}. At time $j$, we draw the plot of the permutation $v_j$ with the points corresponding to the numbers in $\{5,6,7,8,9,10\}$ colored {\color{Teal}teal}; we also draw the TASEP state $\iota(v_j)$ just below the plot. As teal points move to the left in the plots, particles move to the right in the TASEP. In each TASEP state, the positions at or to the right of position $n+1-k'=6$ are shaded {\color{Magenta}magenta}. The particles occupying these positions correspond to the points in the plot lying in the {\color{Magenta}magenta} rectangular region. Note that the number of particles in these positions can only change at times $T+1=2,3,4,5=T'$. Moreover, the number of points in that rectangular region at time $j$ is $(x-y+\h_{v_j}(x,y))n$.  
\end{example} 

\begin{figure}[]
  \begin{center}
\includegraphics[width=\linewidth]{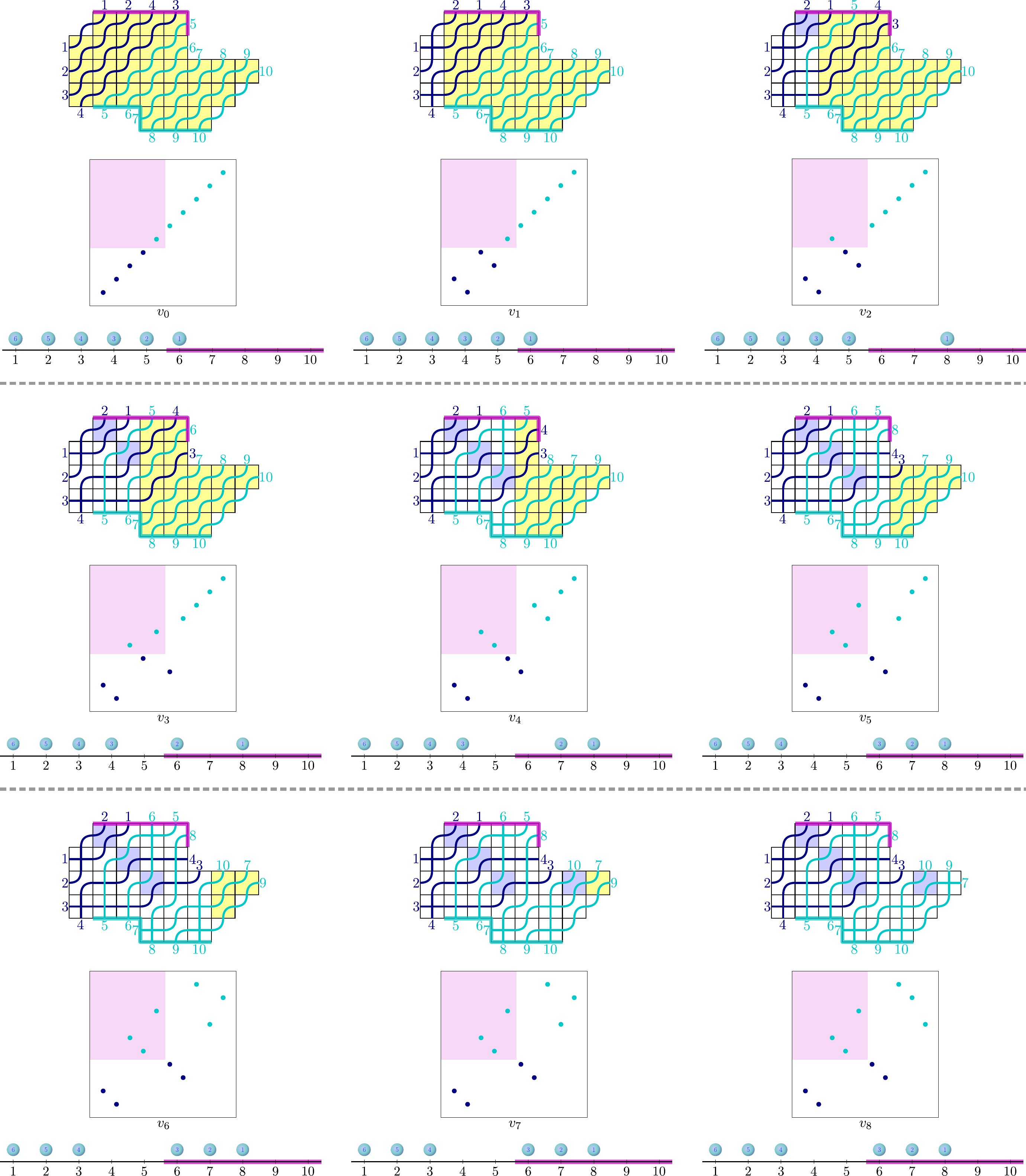}
  \end{center}
\caption{We compute the Demazure product from the shape in \cref{fig:TASEP_example_1} column by column. Under the plot of each permutation $v_j$ is the TASEP state $\iota(v_j)$. }\label{fig:TASEP_example_2}  
\end{figure}

Let us now consider the setup of \cref{thm:NewMPPY}. We have functions $\varphi,\psi\in\mathfrak B$ satisfying $\varphi(z)\leq\psi(z)$ for all $z\in[0,1]$. For simplicity, let us write $\hh_p=\hh_p^{\varphi,\psi}$. For each $n\geq 1$, we have lattice paths $\Lambda_{\swarrow}^{(n)}$ and $\Lambda_{\nearrow}^{(n)}$ in $\boldsymbol{\Lambda}_n$ such that the former lies weakly below the latter. We assume the functions  $\theta_{\Lambda_{\swarrow}^{(n)}}$ and $\theta_{\Lambda_{\nearrow}^{(n)}}$ converge pointwise to $\varphi$ and $\psi$, respectively. Let \[{u_n=\Delta_p\left(\sS\left(\Lambda_{\swarrow}^{(n)},\Lambda_{\nearrow}^{(n)}\right)\right)}.\]  

Let ${(\xx,\yy)\in[0,1]^2}$, and consider the TASEP with $\left\lfloor n(1-y)\right\rfloor$ particles and geometric jumps. For $h\in\mathbb R$, we are interested in the probability \[\mathbb P(\h_{u_n}(x,y)\geq h)=\mathbb P(x-y+\h_{u_n}(x,y)\geq x-y+h).\] This probability is $1$ if $h<0$, and it is $0$ if $h>\min\{1-x,y\}$. Now assume $h\in[0,\min\{1-x,y\}]$. Then the preceding analysis tells us that as $n\to\infty$, the desired probability is asymptotically equal to the probability that at least $(x-y+h)n$ particles have passed to the right of position $n(1-x)$ by time step $n(\psi(x)-\varphi(y))$. If $h<y-x$, this asymptotic probability is $1$; otherwise, we have  
\begin{align}
\nonumber \lim_{n\to\infty}\mathbb P(\h_{u_n}(x,y)\geq h)&=\lim_{n\to\infty}\mathbb{P}\left(\xi_{\left\lfloor n(x-y+h)\right\rfloor}(n(\psi(x)-\varphi(y)))\geq n(1-x)\right) \\ 
\nonumber &=\lim_{n\to\infty}\mathbb{P}\left(\overline \xi_{\left\lfloor n(x-y+h)\right\rfloor}(n(\psi(x)-\varphi(y)))\geq n(1-x)-\xi_{\left\lfloor n(x-y+h)\right\rfloor}(0)\right) \\ 
\label{eq:lim_x-y+h}&=\lim_{n\to\infty}\mathbb{P}\left(\overline \xi_{\left\lfloor n(x-y+h)\right\rfloor}(n(\psi(x)-\varphi(y)))\geq hn\right). 
\end{align} 
If $\cc_p(x-y+h,\psi(x)-\varphi(y))>h$, then \cref{thm:TASEP} implies that the limit in \eqref{eq:lim_x-y+h} is $1$. On the other hand, if $\cc_p(x-y+h,\psi(x)-\varphi(y))<h$, then the limit in \eqref{eq:lim_x-y+h} is $0$. If we fix $x$ and $y$ and increase $h$, then $\cc_p(x-y+h,\psi(x)-\varphi(y))$ decreases while $h$ increases. This implies that 
\[
\lim_{n\to\infty}\h_{u_n}(\xx,\yy)=\sup\{h\in[\max\{0,y-x\},\min\{1-x,y\}]:\cc_p(x-y+h,\psi(x)-\varphi(y))>h\}
\] with probability $1$. If this limit is some number $h^*$ strictly between $\max\{0,y-x\}$ and $\min\{1-x,y\}$, then we have 
\begin{equation}\label{eq:h*}
\cc_p(x-y+h^*,\psi(x)-\varphi(y))=h^*,
\end{equation} and one can solve this equation explicitly to find that $h^*=\ff_p(x,y)$, where we write $\ff_p$ instead of $\ff_p^{\varphi,\psi}$ for convenience (recall the definition from \eqref{eq:f}). This shows that 
\begin{equation}\label{eq:h_limit}
\lim_{n\to\infty}\h_{u_n}(x,y)=\min\{\max\{0,y-x,\ff_p(x,y)\},1-x,y\}=\hh_p(x,y)
\end{equation} with probability $1$, which completes the proof of the second statement in \cref{thm:NewMPPY}. The first statement concerning the existence of the permuton $\zeta_p^{\DD}$ follows immediately since a limit of permuton height functions is a permuton height function. 

We now wish to prove \cref{thm:fluctuations}, which concerns the fluctuations of $\h_{u_n}(x,y)$ as $n\to\infty$. Suppose $(x,y)$ belongs to the set $\mathscr K^{\varphi,\psi}$ defined in \eqref{eq:K}. This implies that \[\hh_p(x,y)=\cc_p(x-y+\hh_p(x,y),\psi(x)-\varphi(y))\] 
For $R\in\mathbb R$, it follows from \cref{thm:TASEP} and \eqref{eq:lim_x-y+h} that  
\begin{align*}
&\phantom{{}={}}\lim_{n\to\infty}\mathbb P\left(\h_{u_n}(x,y)\geq\hh_p(x,y)-Rn^{-2/3}\right) \\ 
&=\lim_{n\to\infty}\mathbb P\left(\overline\xi_{\left\lfloor n(x-y+\hh_p(x,y))-Rn^{1/3}\right\rfloor}\left( n(\psi(x)-\varphi(y))\right)\geq n\hh_p(x,y)-Rn^{1/3}\right) \\ 
&=\lim_{n\to\infty}\mathbb P\left(\overline\xi_{\left\lfloor n(x-y+\hh_p(x,y))-Rn^{1/3}\right\rfloor}\left( n(\psi(x)-\varphi(y))\right)\geq n\cc_p(x-y+\hh_p(x,y),\psi(x)-\varphi(y))-Rn^{1/3}\right) 
\\ &= F_2\left(R\left(\frac{1}{\mathsf{q}_p(x-y+\hh_p(x,y),\psi(x)-\varphi(y))}+\frac{1}{\mathsf{r}_p(x-y+\hh_p(x,y),\psi(x)-\varphi(y))}\right)\right) \\ 
&= F_2(R/\vv_p(x-y+\hh_p(x,y),\psi(x)-\varphi(y))), 
\end{align*}
where the last equality follows from \eqref{eq:v} and \eqref{eq:qr} (with ${\sf m}=x-y+\hh_p(x,y)$ and ${\sf t}=\psi(x)-\varphi(y)$) along with basic algebraic manipulations. Setting $r=R/\vv_p(x-y+\hh_p(x,y),\psi(x)-\varphi(y))$, this completes the proof of \cref{thm:fluctuations}. 

Note that $\mathscr K^{\varphi,\psi}$ is the set of points $(x,y)\in[0,1]^2$ such that the supremum in \eqref{eq:h_limit} is some number $h^*$ strictly between $\max\{0,y-x\}$ and $\min\{1-x,y\}$. It follows that \[[0,1]^2\setminus\mathscr K^{\varphi,\psi}=\PP_{\searrow}^{\varphi,\psi}\cup\PP_{\swarrow}^{\varphi,\psi}\cup\PP_{\nearrow}^{\varphi,\psi}\cup\PP_{\nwarrow}^{\varphi,\psi},\]
where 
\begin{align*}
\PP_{\searrow}^{\varphi,\psi}&:=\{(x,y)\in[0,1]^2:\cc_p(x-y,\psi(x)-\varphi(y))=0\}; \\ 
\PP_{\swarrow}^{\varphi,\psi}&:=\{(x,y)\in[0,1]^2:\cc_p(x,\psi(x)-\varphi(y))\geq y\}; \\ 
\PP_{\nearrow}^{\varphi,\psi}&:=\{(x,y)\in[0,1]^2:\cc_p(1-y,\psi(x)-\varphi(y))\geq 1-x\}; \\ 
\PP_{\nwarrow}^{\varphi,\psi}&:=\{(x,y)\in[0,1]^2:\cc_p(0,\psi(x)-\varphi(y))\leq y-x\}.  
\end{align*}
The first and last of these regions have even more explicit descriptions. Namely, 
\begin{equation}\label{eq:simplified_R}
\PP_{\searrow}^{\varphi,\psi}=\{(x,y)\in[0,1]^2:p(\psi(x)-\varphi(y))\leq x-y\}\end{equation} 
and 
\begin{equation}\label{eq:simplified_R2}
\PP_{\nwarrow}^{\varphi,\psi}=\{(x,y)\in[0,1]^2:p(\psi(x)-\varphi(y))/(1-p)\leq y-x\}.
\end{equation} 

\begin{remark}\label{rem:symmetry}
It will be helpful to record a certain symmetry. Given a permuton $\mu$, let $\widehat \mu$ be the permuton obtained by ``rotating'' $\mu$ by $180^\circ$. More precisely, $\widehat\mu$ is the permuton such that $\h_{\widehat\mu}(x,y)=x-y+\h_{\mu}(x,y)$ for all $(x,y)\in[0,1]^2$. Now suppose $\sS$ is a shape whose boxes have contents lying in $[n-1]$, and let $\widehat\sS$ be the shape obtained by reflecting $\sS$ through the line $\mathcal L_{n/2}$. If ${\bf w}(\sS)=(i_1,i_2,\ldots,i_k)$, then ${\bf w}(\widehat{\mathscr{S}}\,\,)$ is commutation equivalent to $(n-i_1,n-i_2,\ldots,n-i_k)$. This implies that $\Delta_p(\widehat \sS\,\,)$ has the same distribution as $\delta_n\Delta_p(\widehat{\sS}\,\,)\delta_n$, where $\delta_n$ is the decreasing permutation in $\SSS_n$. It follows from \cref{thm:NewMPPY} that this symmetry passes to limiting permutons when we take the scaling limits of shapes. More precisely, we have 
\[\zeta_p^{\widehat\DD}=\widehat{\zeta_p^{\DD}}\] for every $\DD\in\RR$, where $\widehat\DD$ is the region in $\RR$ obtained by reflecting $\DD$ through the line $\mathcal L_{1/2}$. 
\end{remark} 

\subsection{Rectangle Shapes and Peridot Permutons}\label{subsec:gems} 

Let us now work through an explicit example in detail. Fix $0<\beta<1$. We will take the scaling limit $\DD$ of our shapes to be an axis-parallel rectangle of height $1-\beta$ and width $\beta$, which we denote by $\Rect^\beta$. 
Thus, 
\[\varphi(z)=\begin{cases}
    0 & \text{if $0\leq z\leq1-\beta$}; \\
    z-1+\beta & \text{if $1-\beta\leq z\leq 1$} 
\end{cases}\quad\text{and}\quad \psi(z)=\begin{cases}
    z & \text{if $0\leq z\leq \beta$}; \\
    \beta & \text{if $\beta\leq z\leq 1$}.  
\end{cases}\]
We can take the path $\Lambda_{\swarrow}^{(n)}$ to consist of $\left\lfloor(1-\beta) n\right\rfloor$ south steps followed by $\left\lceil\beta n\right\rceil$ east steps. Similarly, we take the path $\Lambda_{\nearrow}^{(n)}$ to consist of $\left\lceil\beta n\right\rceil$ east steps followed by $\left\lfloor(1-\beta) n\right\rfloor$ south steps. 

For $(x,y)\in[0,1]^2$, we have 
\[\psi(x)-\varphi(y)=\begin{cases}
    x & \text{if $(x,y)\in[0,\beta]\times[0,1-\beta]$}; \\
    x-y+1-\beta & \text{if $(x,y)\in[0,\beta]\times[1-\beta,1]$}; \\
    \beta & \text{if $(x,y)\in[\beta,1]\times[0,1-\beta]$}; \\
    1-y & \text{if $(x,y)\in[\beta,1]\times[1-\beta,1]$}. 
\end{cases}\]
Note that $\psi(x)-\varphi(y)\leq \min\{x,1-y\}$. It follows that $\cc_p(x,\psi(x)-\varphi(y))=\cc_p(1-y,\psi(x)-\varphi(y))=0$, so \[\PP_{\swarrow}^{\varphi,\psi}=\{(x,y)\in[0,1]^2:y=0\}\quad\text{and}\quad \PP_{\nearrow}^{\varphi,\psi}=\{(x,y)\in[0,1]^2:x=1\}.\] 

By \eqref{eq:simplified_R}, the region $\PP_{\searrow}^{\varphi,\psi}$ is determined by the inequality $p(\psi(x)-\varphi(y))\leq x-y$. The northwest boundary of $\PP_{\searrow}^{\varphi,\psi}$ is the piecewise-linear curve that starts at $(0,0)$, travels through \[[0,\beta]\times[0,1-\beta]\] along a line segment of slope $1-p$, then travels through \[{([0,\beta]\times[1-\beta,1])\cup([\beta,1]\times[0,1-\beta])}\] along a line segment of slope $1$, and then travels through \[[\beta,1]\times[1-\beta,1]\] along a line segment of slope $1/(1-p)$ until reaching $(1,1)$. The middle line segment travels with slope $1$ through $[0,\beta]\times[1-\beta,1]$ if $\beta\geq 1/(2-p)$, and it travels with slope $1$ through $[\beta,1]\times[0,1-\beta]$ if $\beta\leq 1/(2-p)$. (If $\beta=1/(2-p)$, then the middle line segment is just the single point $(\beta,1-\beta)$.)  

By \eqref{eq:simplified_R2}, the region $\PP_{\nwarrow}^{\varphi,\psi}$ is determined by the inequality $p(\psi(x)-\varphi(y))/(1-p)\leq y-x$. The southeast boundary of $\PP_{\nwarrow}^{\varphi,\psi}$ is the piecewise-linear curve that starts at $(0,0)$, travels through \[{[0,\beta]\times[0,1-\beta]}\] along a line segment of slope $1/(1-p)$, then travels through \[{([0,\beta]\times[1-\beta,1])\cup([\beta,1]\times[0,1-\beta])}\] along a line segment of slope $1$, and then travels through \[[\beta,1]\times[1-\beta,1]\] along a line segment of slope $1-p$ until reaching $(1,1)$. The middle line segment travels with slope $1$ through $[0,\beta]\times[1-\beta,1]$ if $\beta\geq(1-p)/(2-p)$, and it travels with slope $1$ through $[\beta,1]\times[0,1-\beta]$ if $\beta\leq (1-p)/(2-p)$. (If $\beta=(1-p)/(2-p)$, then the middle line segment is just the single point $(\beta,1-\beta)$.) 

\begin{figure}[]
  \begin{center}
  \includegraphics[height=5cm]{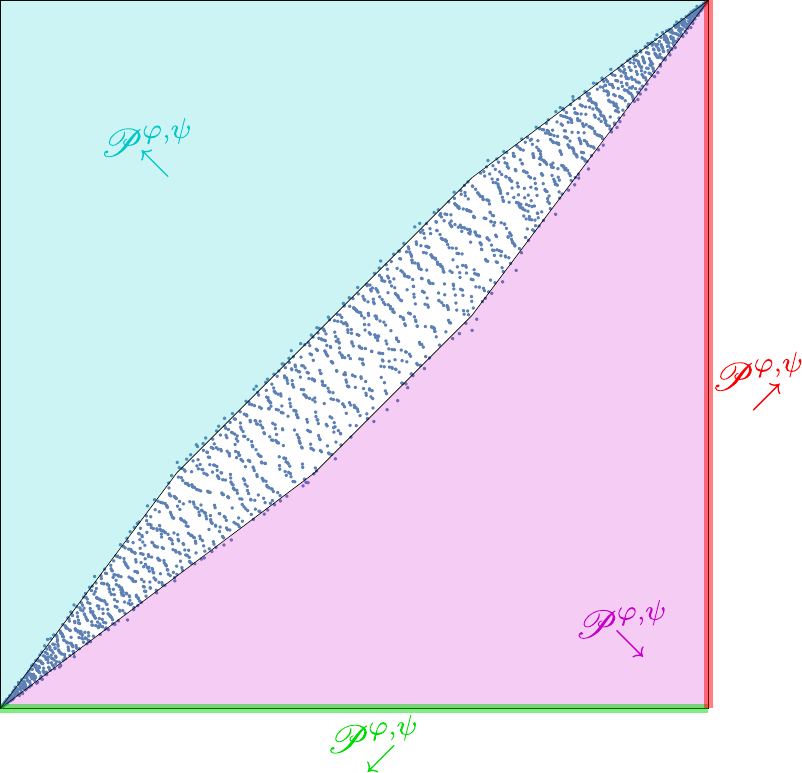}\quad\includegraphics[height=5cm]{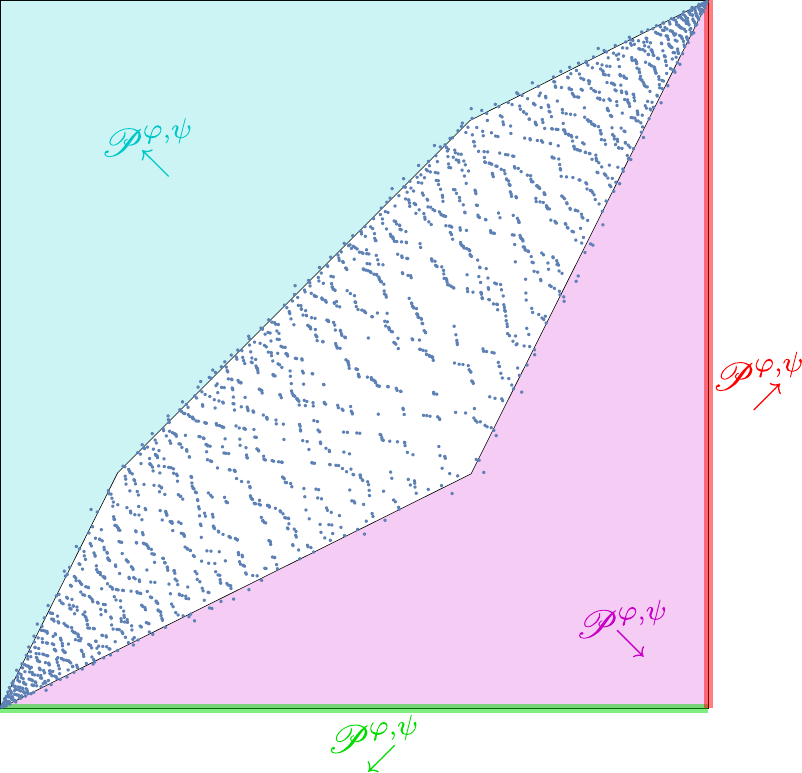}\quad\includegraphics[height=5cm]{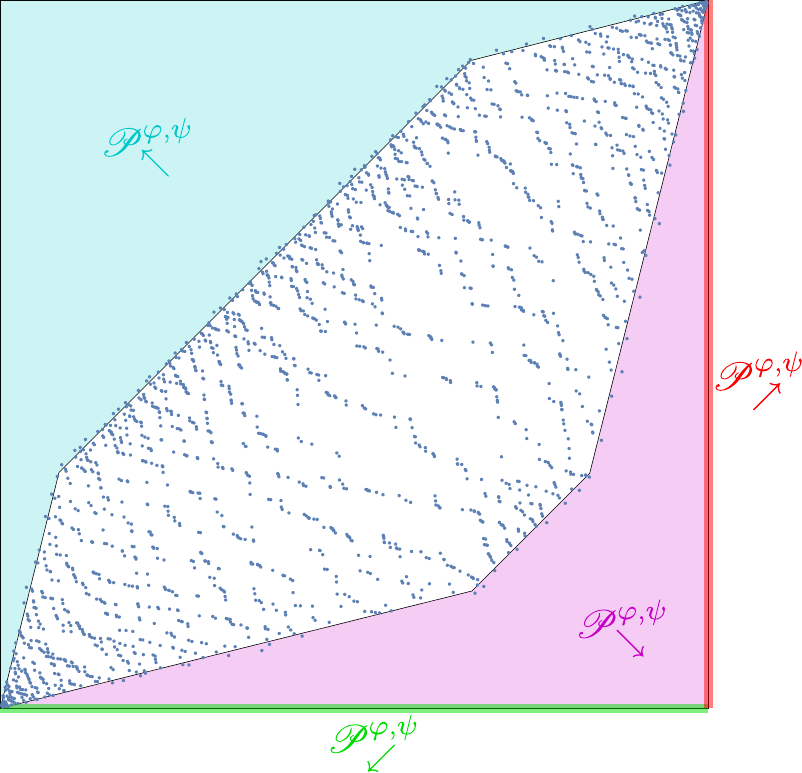}
  \end{center}
\caption{Plots of permutations $\Delta_{1/4}(\sS)$ (left), $\Delta_{1/2}(\sS)$ (middle), and $\Delta_{3/4}(\sS)$ (right) in $\SSS_{2400}$, where $\sS$ is an $800\times 1600$ rectangle shape. Each plot approximates a peridot permuton with parameter ${\beta=2/3}$. We have shaded the regions $\PP_{\searrow}^{\varphi,\psi},\PP_{\swarrow}^{\varphi,\psi},\PP_{\nearrow}^{\varphi,\psi},\PP_{\nwarrow}^{\varphi,\psi}$.}\label{fig:gems}  
\end{figure}

The support of the permuton $\zeta_p^{\Rect^\beta}$ is the closure of $\mathscr K^{\varphi,\psi}=[0,1]^2\setminus(\PP_{\searrow}^{\varphi,\psi}\cup\PP_{\swarrow}^{\varphi,\psi}\cup\PP_{\nearrow}^{\varphi,\psi}\cup\PP_{\nwarrow}^{\varphi,\psi})$. We call $\zeta_p^{\Rect^\beta}$ a \dfn{peridot permuton}; see \cref{fig:gems}.

\subsection{Trapezoid Shapes and Polyphemus Permutons}\label{subsec:malts} 
For our second explicit example, we take the scaling limit $\DD$ of our shapes to be a trapezoid with two opposite sides of slope $1$. Fix $\alpha,\beta\in[0,1]$ and $R\geq \max\{0,\beta-\alpha\}$, and define \[\varphi(z)=\alpha z\quad\text{and}\quad \psi(z)=\beta z+R.\]
The limiting trapezoidal shape $\DD$ has vertices $(0,0)$, $(R,R)$, $(\beta+R,\beta+R-1)$, and $(\alpha,\alpha-1)$. 

For $(x,y)\in[0,1]^2$, we have 
\[\psi(x)-\varphi(y)=\beta x-\alpha y+R.\]
By \eqref{eq:simplified_R}, the region $\PP_{\searrow}^{\varphi,\psi}$ is determined by the inequality $p(\psi(x)-\varphi(y))\leq x-y$, which simplifies to the linear inequality $(1-p\alpha)y+pR\leq (1-p\beta)x$. Similarly, by \eqref{eq:simplified_R2}, the region $\PP_{\nwarrow}^{\varphi,\psi}$ is determined by the inequality $p(\psi(x)-\varphi(y))/(1-p)\leq y-x$, which simplifies to the linear inequality $(1-p(1-\beta))x+pR\leq (1-p(1-\alpha))y$. 

The equation 
\[
(x-y+p(\beta x+(1-\alpha)y+R))^2=4px(\beta x-\alpha y+R)
\] defines a conic that is tangent to the $y$-axis and the $x$-axis at the points 
\begin{equation}\label{eq:tangent1}
\left(0,\frac{pR}{1-p(1-\alpha)}\right)\quad \text{and}\quad\left(\frac{pR}{1-p\beta},0\right),
\end{equation} respectively. Let $\mathscr{E}_{\swarrow}$ denote the concave-up arc of this conic whose endpoints are the points in \eqref{eq:tangent1}. (When $R=0$, the curve $\mathscr{E}_{\swarrow}$ degenerates to the single point $(0,0)$.) 
The region $\PP_{\swarrow}^{\varphi,\psi}$ consists of the points $(x,y)\in[0,1]^2$ that lie weakly below $\mathscr{E}_{\swarrow}$ or satisfy $y=0$. 

The equation 
\[
(x-y+p(1+R-(1-\beta) x-\alpha y))^2=4p(1-y)(\beta x-\alpha y+R)
\] defines a conic that is tangent to the lines $\{(1,y):y\in\mathbb R\}$ and $\{(x,1):x\in\mathbb R\}$ at the points 
\begin{equation}\label{eq:tangent2}
\left(1,\frac{1-p\beta-pR}{1-p\alpha}\right)\quad \text{and}\quad\left(\frac{1-p(1-\alpha)-pR}{1-p(1-\beta)},1\right),
\end{equation} respectively. Let $\mathscr{E}_{\nearrow}$ denote the concave-down arc of this conic whose endpoints are the points in \eqref{eq:tangent2}. (When $R=\alpha-\beta$, the curve $\mathscr{E}_{\nearrow}$ degenerates to the single point $(1,1)$.) 
The region $\PP_{\nearrow}^{\varphi,\psi}$ consists of the points $(x,y)\in[0,1]^2$ that lie weakly above $\mathscr{E}_{\nearrow}$ or satisfy $x=1$. 

If \[R\geq\frac{2-(1-\alpha+\beta)p+\sqrt{(1-\alpha-\beta)^2p^2+4(1-p)}}{2p},\] then $\mathscr{E}_{\swarrow}$ lies weakly above the line $\{(x,1-x):x\in\mathbb R\}$ while $\mathscr{E}_{\nearrow}$ lies weakly below this line. In this case, it follows that $\PP_{\swarrow}^{\varphi,\psi}\cap\PP_{\nearrow}^{\varphi,\psi}$ contains the entire line segment $\{(x,1-x):x\in[0,1]\}$, and it is straightforward to show that the permuton $\zeta_p^\DD$ is the uniform measure on this line segment (i.e., the permuton limit of decreasing permutations). 

Now assume \[R<\frac{2-(1-\alpha+\beta)p+\sqrt{(1-\alpha-\beta)^2p^2+4(1-p)}}{2p}.\] Let $\mathcal V_{\LLT}$ be the union of the sides of the triangle with vertices $(0,0)$, $(0,1)$, and $(1,0)$. Let $\mathcal V_{\URT}$ be the union of the sides of the triangle with vertices $(1,1)$, $(0,1)$, and $(1,0)$. 
The intersection $\mathscr E_{\swarrow}\cap\mathcal V_{\LLT}$ contains two points $({\sf x}_\swarrow^{(1)},{\sf y}_\swarrow^{(1)})$ and $({\sf x}_\swarrow^{(2)},{\sf y}_\swarrow^{(2)})$ satisfying ${\sf x}_{\swarrow}^{(1)}\leq {\sf x}_{\swarrow}^{(2)}$ and ${\sf y}_{\swarrow}^{(1)}\geq {\sf y}_{\swarrow}^{(2)}$; these points coincide if and only if $R=0$, in which case they are the point $(0,0)$. The intersection $\mathscr E_{\nearrow}\cap\mathcal V_{\URT}$ contains two points $({\sf x}_{\nearrow}^{(1)},{\sf y}_{\nearrow}^{(1)})$ and $({\sf x}_{\nearrow}^{(2)},{\sf y}_{\nearrow}^{(2)})$ satisfying ${\sf x}_{\nearrow}^{(1)}\leq {\sf x}_{\nearrow}^{(2)}$ and ${\sf y}_{\nearrow}^{(1)}\geq {\sf y}_{\nearrow}^{(2)}$; these points coincide if and only if $R=\alpha-\beta$, in which case they are the point $(1,1)$. If $pR>1-p(1-\alpha)$, then the points $({\sf x}_{\swarrow}^{(1)},{\sf y}_{\swarrow}^{(1)})$ and $({\sf x}_{\nearrow}^{(1)},{\sf y}_{\nearrow}^{(1)})$ coincide and lie on the line segment with endpoints $(0,1)$ and $(1,0)$; in this case, the permuton $\mu^{\varphi,\psi}$ has a singular part that is the uniform measure on the line segment with endpoints $(0,1)$ and $({\sf x}_{\swarrow}^{(1)},{\sf y}_{\swarrow}^{(1)})$. Similarly, if $pR>1-p\beta$, then the points $({\sf x}_{\swarrow}^{(2)},{\sf y}_{\swarrow}^{(2)})$ and $({\sf x}_{\nearrow}^{(2)},{\sf y}_{\nearrow}^{(2)})$ coincide and lie on the line segment with endpoints $(0,1)$ and $(1,0)$; in this case, the permuton $\mu^{\varphi,\psi}$ has a singular part that is the uniform measure on the line segment with endpoints $(1,0)$ and $({\sf x}_{\swarrow}^{(2)},{\sf y}_{\swarrow}^{(2)})$. 
The permuton $\zeta_p^\DD$ also has a singular part on 
\[\left(\mathscr{E}_{\swarrow}\cap\{(x,y)\in[0,1]^2:y\leq 1-x\}\right)\cup\left(\mathscr{E}_{\nearrow}\cap\{(x,y)\in[0,1]^2:y\geq 1-x\}\right).\] As in \cite[Propositions~4.2~and~4.4]{MPPY}, one can in principal compute the mass on this singular part explicitly by considering the partial derivatives of the height function $\hh_p=\h_{\zeta_p^\DD}$.  

When $\alpha=0$, $\beta=1$, and $R=0$, the curve $\mathscr E_{\swarrow}$ degenerates to the single point $(0,0)$, and the limiting permuton $\zeta_p^\DD$ is precisely the Grothendieck permuton (with parameter $p$) from \cite{MPPY}. In general, if we fix $p,\alpha,\beta$ and gradually increase $R$ (starting at $\min\{0,\alpha-\beta\}$), the shape of the permuton will transform from a shape resembling an ice cream cone to a shape resembling a closing eye. We call $\zeta_p^\DD$ a \dfn{Polyphemus permuton}; see \cref{fig:trapezoid_example_plot}.

\begin{figure}[]
  \begin{center}
  \includegraphics[height=5cm]{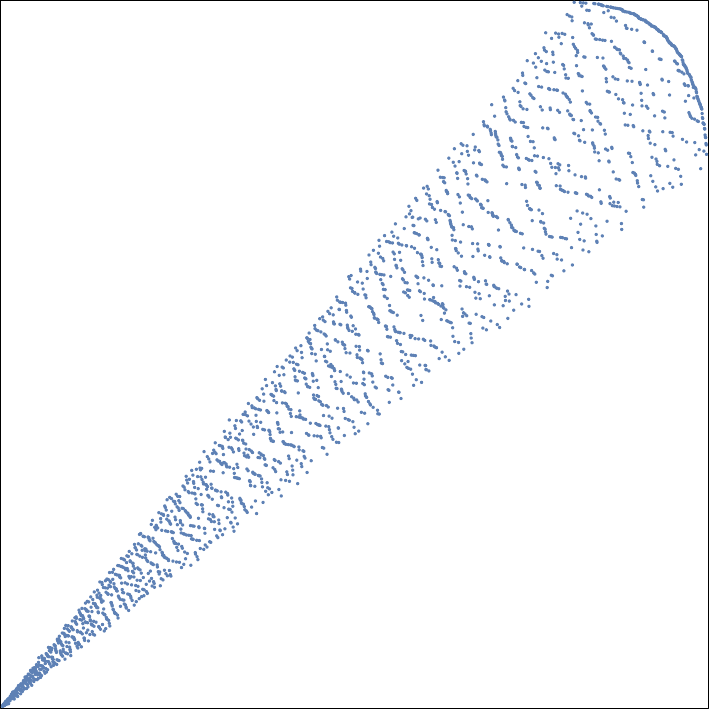}\quad\includegraphics[height=5cm]{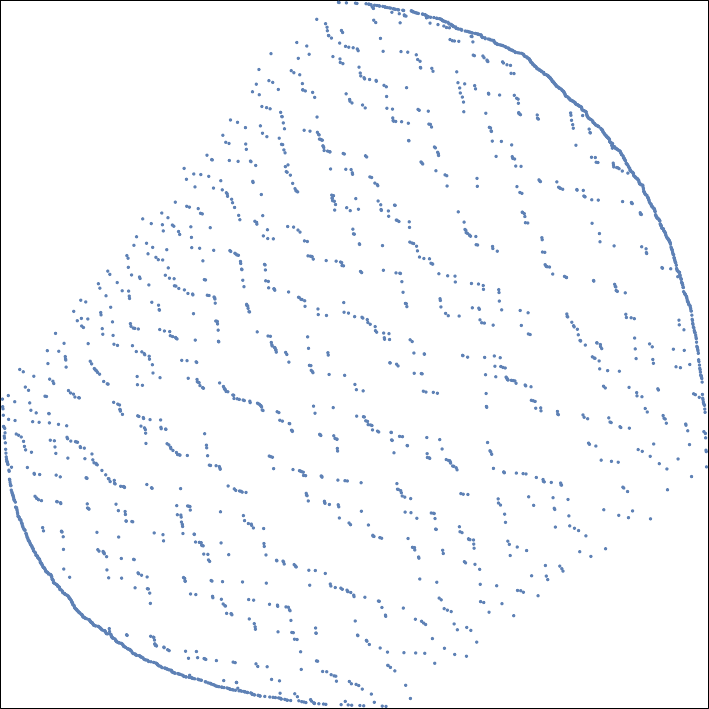}\quad\includegraphics[height=5cm]{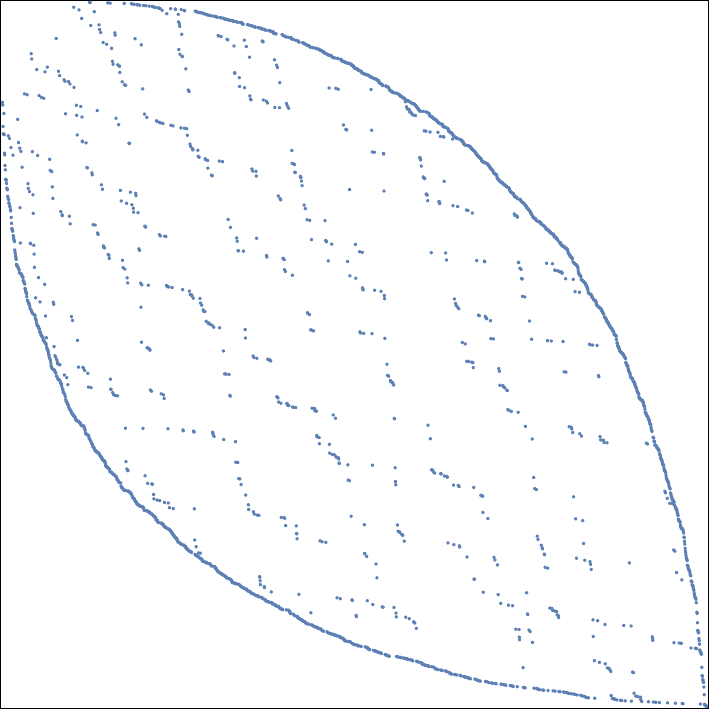} \\ \vspace{0.3cm}
    \includegraphics[height=5cm]{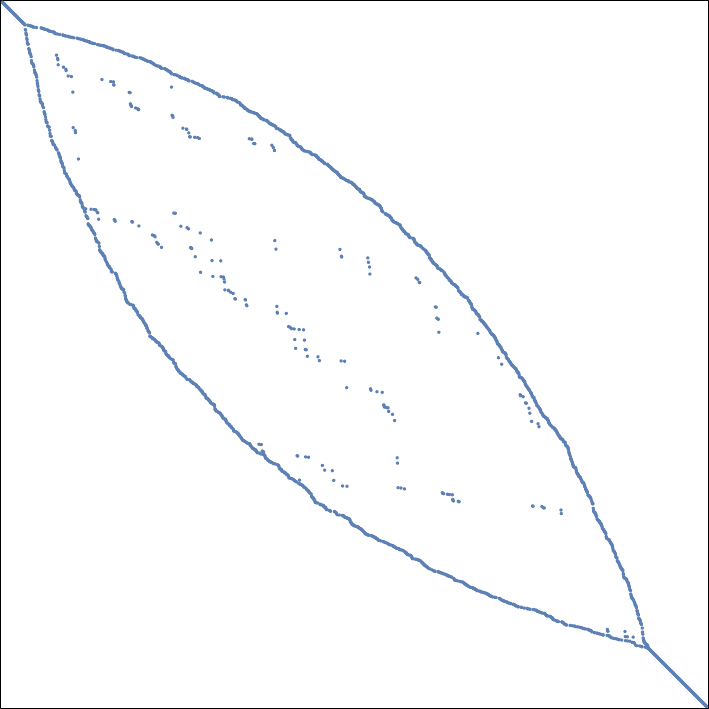}\quad\includegraphics[height=5cm]{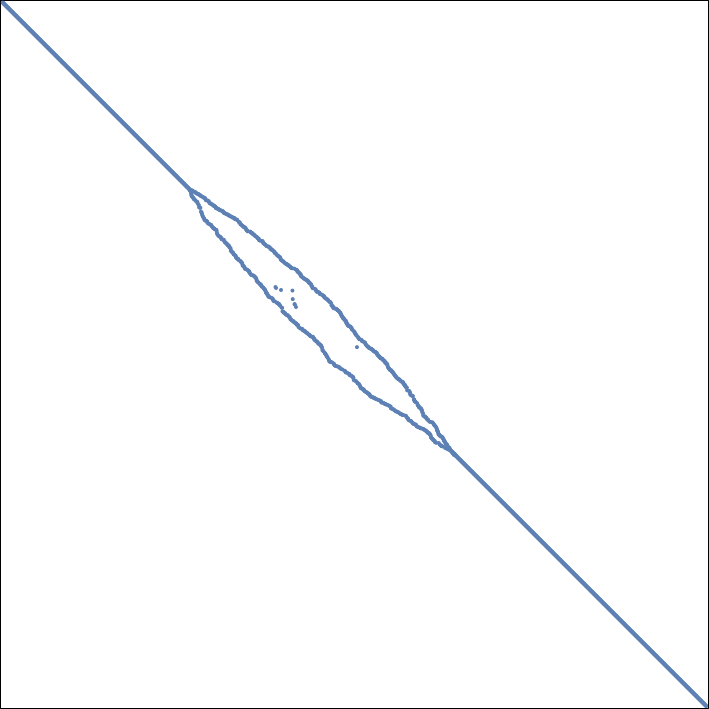}\quad\includegraphics[height=5cm]{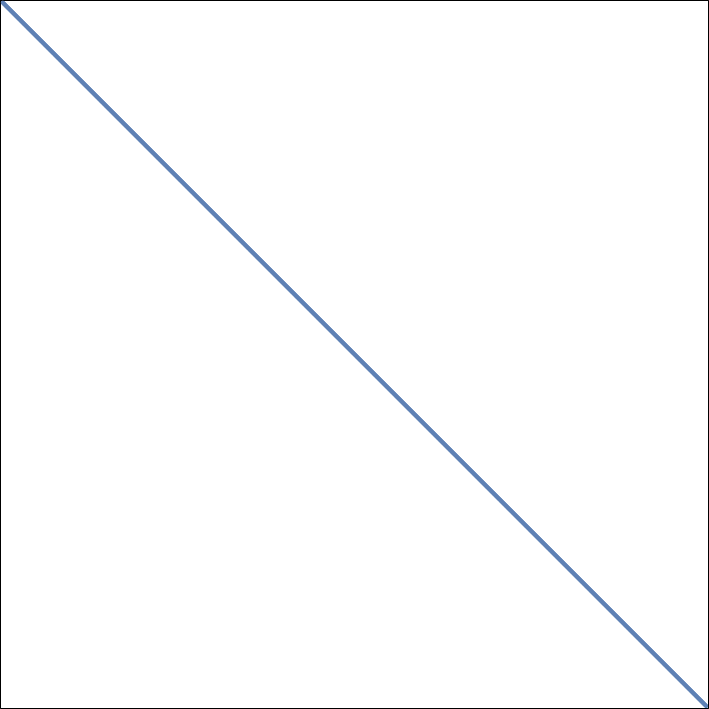}
  \end{center}
\caption{Plots of permutations in $\SSS_{2400}$ that approximate Polyphemus permutons with parameters $p=3/5$, $\alpha=1/2$, $\beta=3/4$. The values of $R$ in the top row are $0,1/2,1$, while the values of $R$ in the bottom row are $3/2,2,5/2$. }\label{fig:trapezoid_example_plot}  
\end{figure}

\subsection{Non-Order-Convex Shapes }\label{subsec:penguins} 

Suppose $\sS$ is a shape that is not order-convex. We can still consider the Demazure product $\Delta_p(\sS)$ of the random subword $\sub_p({\bf w}(\sS))$. This Demazure product can also be computed graphically using pipe dreams. As before, we fill the boxes in $\sS$ independently at random so that each box is filled with a cross tile $\Cross$ with probability $p$ and with a bump tile $\Bump$ with probability $1-p$. The difference is that we now need to add extra {\color{Gold}golden} bump tiles $\GoldBump$ (deterministically) in order to complete $\sS$ to an order-convex shape. This is illustrated in \cref{fig:nonconvex}. We once again resolve crossings when computing the Demazure product. 

\begin{figure}[]
  \begin{center}
  \includegraphics[height=6.375cm]{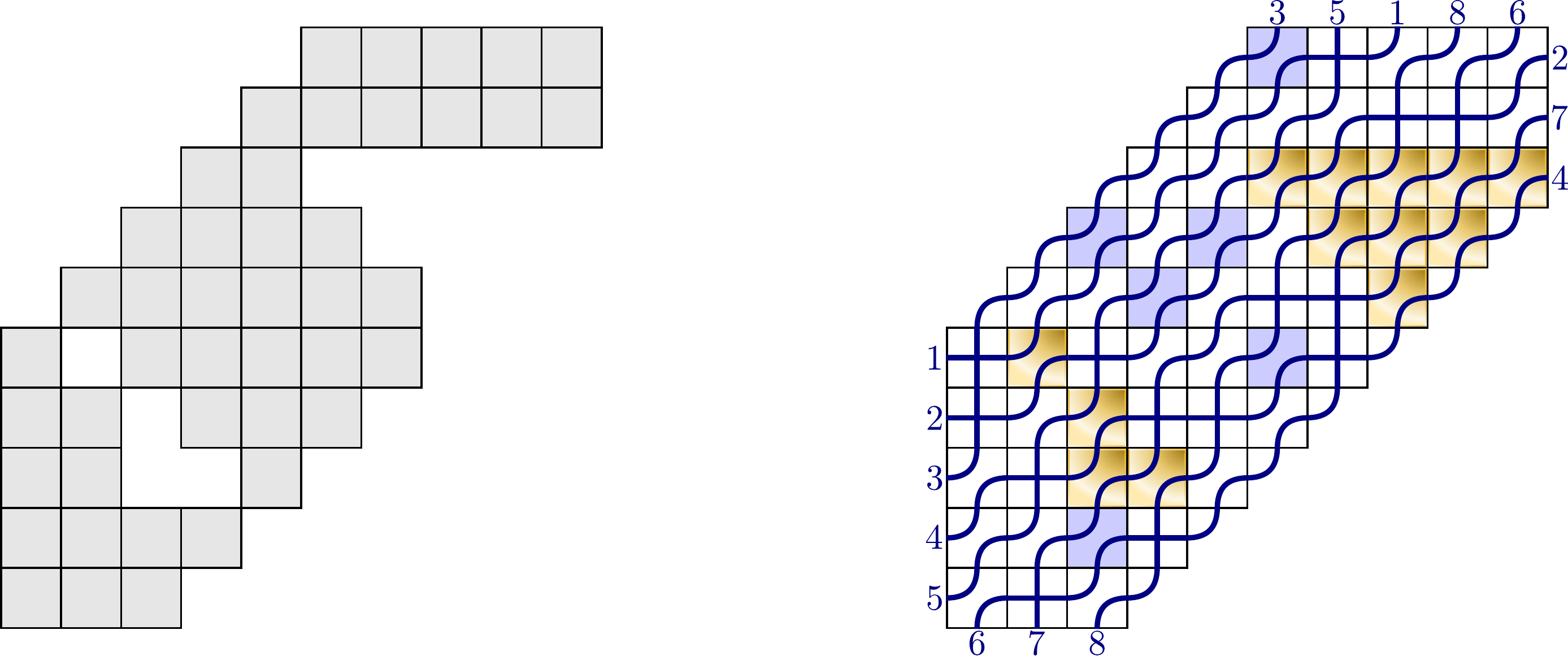}
  \end{center}
\caption{On the left is a shape that is not order-convex; its boxes are lightly shaded. On the right, we have randomly filled the boxes of the shape on the left with cross tiles and bump tiles. We have then added extra ({\color{Gold}golden}) bump tiles to form an order-convex shape. We have then resolved some of the crossings to compute the Demazure product, which is $35186274$. }\label{fig:nonconvex} 
\end{figure}

Now let $(\sS^{(n)})_{n\geq 1}$ be a sequence of shapes, where the contents of the boxes in $\sS^{(n)}$ belong to $[n-1]$. Suppose that $\frac{1}{n}\sS^{(n)}$ converges to some limit shape $\DD$. Let $u_n=\Delta_p(\sS^{(n)})$. We expect the height functions $\h_{u_n}$ to converge to the height function of some limiting permuton $\zeta_p^{\DD}$. However, we cannot directly compute this limiting permuton using the TASEP as we did in the proof of \cref{thm:NewMPPY}. To solve this issue, we employ the tools we developed in \cref{sec:Demazure_permutons}. Let us assume that $\DD$ can be written as $\DD_1\cup\cdots\cup\DD_r$, where the shapes $\DD_1,\ldots,\DD_r$ have disjoint interiors and each belong to the set $\RR$ defined in \eqref{eq:RR}. Assume also that for all $1\leq i<j\leq r$, no point in $\DD_i$ is northeast of a different point in $\DD_j$. By \cref{thm:NewMPPY}, each shape $\DD_i$ has an associated permuton $\zeta_p^{\DD_i}$. It follows from \cref{prop:Demazure_permuton} that \[\zeta_p^\DD=\zeta_p^{\DD_1}\star\cdots\star\zeta_p^{\DD_r}.\]
Since we can explicitly compute the height functions of the permutons $\zeta_p^{\DD_1},\ldots,\zeta_p^{\DD_r}$ using \cref{thm:NewMPPY}, this allows us to explicitly compute the height function of $\zeta_p^{\DD}$. We leave this task to the motivated reader, contenting ourselves here by simply presenting pictures that showcase some of the exotic permutons that can be created in this manner. 

\begin{figure}[]
  \begin{center}
  \includegraphics[height=3cm]{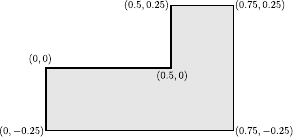} \qquad\qquad\qquad \includegraphics[height=2.778cm]{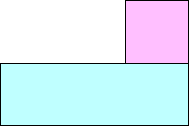} \\ \vspace{0.3cm}
  \includegraphics[height=5cm]{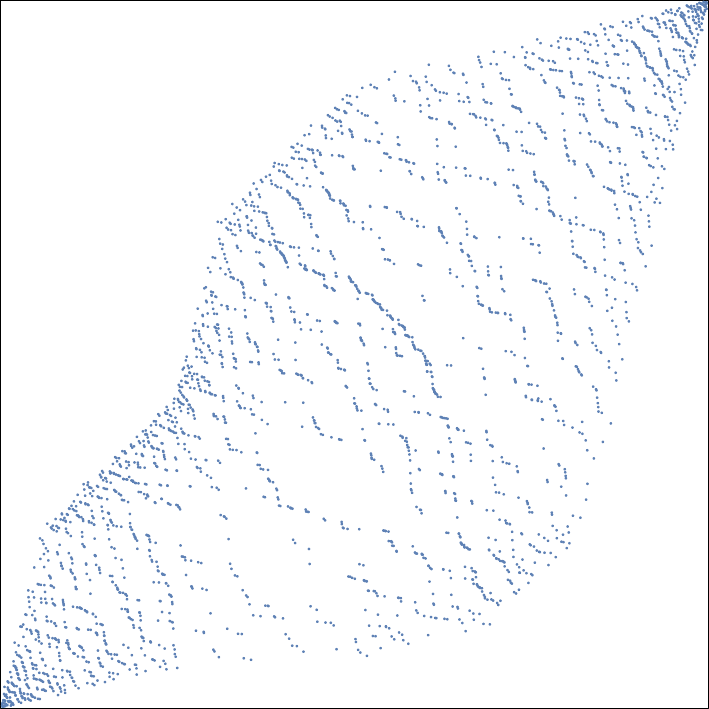}\quad\includegraphics[height=5cm]{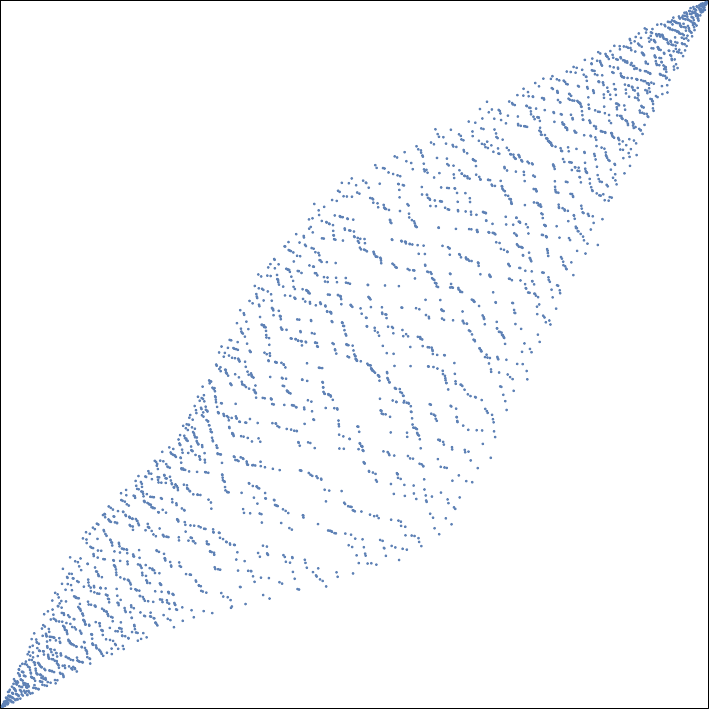}\quad\includegraphics[height=5cm]{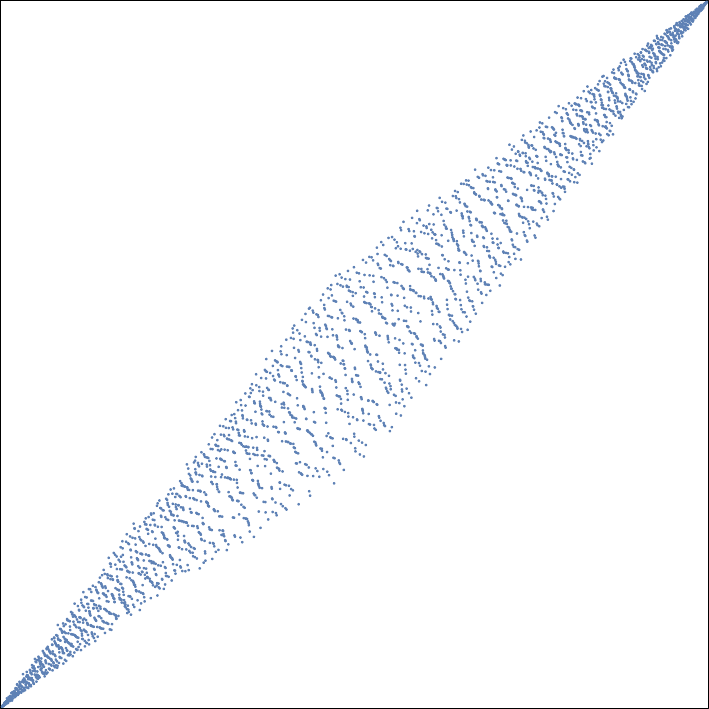} 
  \end{center}
\caption{On the top left is a limit shape $\Le$ that can be decomposed into two rectangles as illustrated on the top right. On the bottom are plots of permutations in $\SSS_{3200}$ that approximate the limiting permutons $\zeta_p^{\Le}$ for $p=1/4$ (left), $p=1/2$ (middle), and $p=3/4$ (right). We call these limiting permutons $\zeta_p^{\Le}$ \dfn{Platyhelminthes permutons}.}\label{fig:L-shape} 
\end{figure}

\begin{example}\label{exam:penguins}
Consider the limit shape $\Le$ shown on the top left of \cref{fig:L-shape}. 
We can decompose this shape into two rectangles as on the top right of \cref{fig:L-shape}. For each $p\in(0,1)$, one could compute the limiting permuton $\zeta_p^{\Le}$ by combining the explicit computations from \cref{subsec:gems} with \cref{prop:Demazure_permuton}. The bottom of \cref{fig:L-shape} shows plots of permutations in $\SSS_{3200}$ that approximate $\zeta_p^{\Le}$ for $p=1/4$, $p=1/2$, and $p=3/4$. 
\end{example}

\begin{figure}[]
  \begin{center}
  \includegraphics[height=2.772cm]{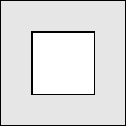} \qquad\qquad\qquad \includegraphics[height=2.778cm]{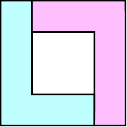} \\ \vspace{0.3cm} 
    \includegraphics[height=5cm]{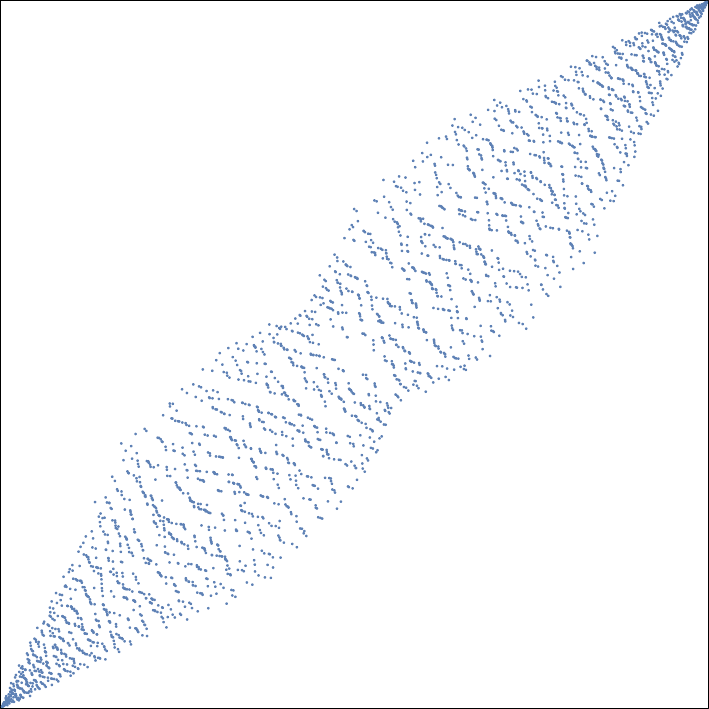}\quad\includegraphics[height=5cm]{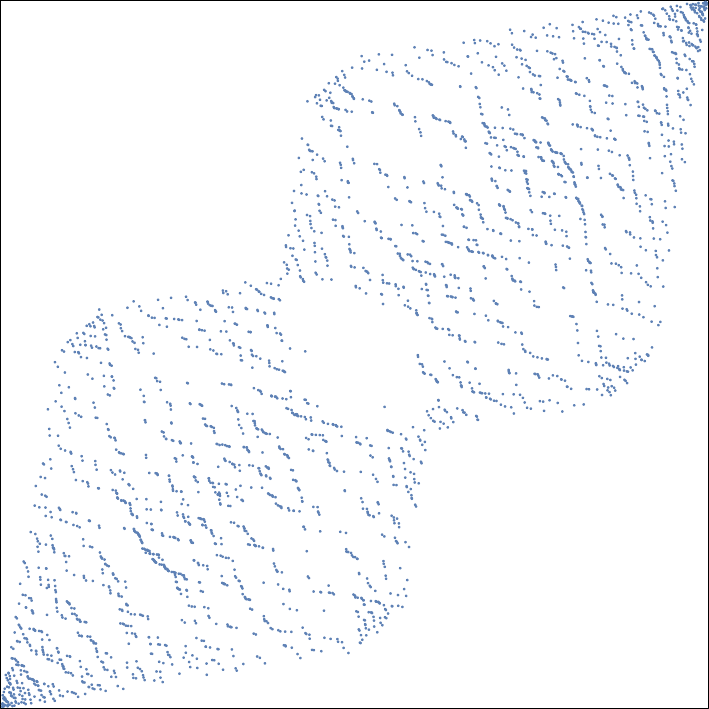}\quad\includegraphics[height=5cm]{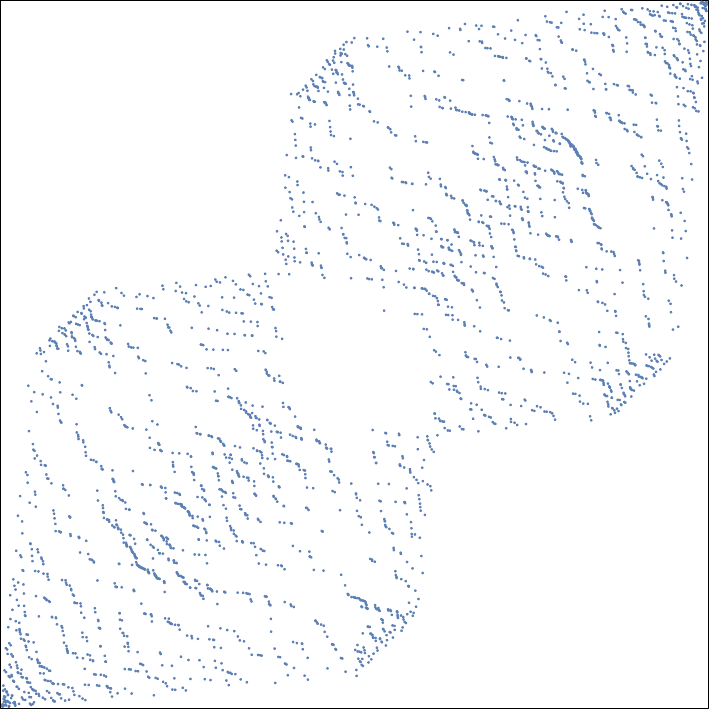} 
  \end{center}
\caption{On the top left is a limit shape $\Hole$ that can be decomposed into two order-convex regions as illustrated on the top right. On the bottom are plots of permutations in $\SSS_{3200}$ that approximate the limiting permutons $\zeta_p^{\Hole}$ for $p=1/2$ (left), $p=5/6$ (middle), and $p=9/10$ (right). We call these limiting permutons $\zeta_p^{\Hole}$ \dfn{pointy peanut permutons}. }\label{fig:hole} 
\end{figure}

\begin{example}\label{exam:knives}
Consider the limit shape $\Hole$ shown on the top left of \cref{fig:hole}. 
We can decompose this shape into two order-convex regions as on the top right of \cref{fig:hole}. For each $p\in(0,1)$, one could compute the limiting permuton $\zeta_p^{\Hole}$ by combining \cref{thm:NewMPPY,prop:Demazure_permuton} (alternatively, one could decompose $\Hole$ into four rectangles and appeal to the explicit computations from \cref{subsec:gems}). The bottom of \cref{fig:hole} shows plots of permutations in $\SSS_{3200}$ that approximate $\zeta_p^{\Hole}$ for $p=1/2$, $p=5/6$, and $p=9/10$. 
\end{example} 

\section{Bubble-Sort Permutons}\label{sec:Bubble} 

With the machinery we have developed so far, we can easily prove \cref{thm:BubbleGeneral}. Let us preserve the notation from the statement of that theorem. Note that $\tau_{{\bf w}(\sS^{(n)})}(v_n)=v_n\star\Delta_1(\sS^{(n)})$. \cref{thm:NewMPPY} tells us that for each point $(x,y)\in[0,1]^2$, we have $\lim_{n\to\infty}\h_{\Delta_1(\sS^{(n)})}(x,y)=\hh_1^{\varphi,\psi}(x,y)$ with probability $1$. The first statement in \cref{thm:BubbleGeneral} now follows directly from \cref{prop:Demazure_permuton}. To deduce the second statement, note that if $(v_n)_{n\geq 1}$ is a sequence of independent permutations such that $v_n$ is chosen uniformly at random from $\SSS_n$, then $\mu$ is the uniform permuton $\upsilon$. In this case, the permuton limit of $(\pi_{u_n})_{n\geq 1}$ is $\upsilon\star\zeta_1^\DD$, where $\DD=\mathscr{R}^{\varphi,\psi}\in\RR$. 

\begin{remark}\label{rem:symmetry2}
Recall from \cref{rem:symmetry} that we write $\widehat\mu$ for the permuton obtained by rotating the permuton $\mu$ by $180^\circ$. Note that the uniform measure $\upsilon$ satisfies $\upsilon=\widehat\upsilon$. For any permutons $\mu$ and $\nu$, we have 
\[\widehat\mu\star\widehat\nu=\widehat{\mu\star\nu}.\]
Therefore, it follows from \cref{rem:symmetry} that 
\[\upsilon\star\zeta_p^{\widehat\DD}=\widehat\upsilon\star\widehat{\zeta_p^\DD}=\widehat{\upsilon\star\zeta_p^\DD}\] for every $\DD\in\RR$ and $p\in(0,1]$.   
\end{remark} 

The remainder of this section is devoted to proving \cref{thm:c-Bubble_Line,thm:rectangle}, which illustrate \cref{thm:BubbleGeneral} in special cases.

\subsection{Parallelogram Shapes}\label{subsec:parallelogram} 

Fix $\alpha>0$ and $\beta\in[0,1]$, and preserve the notation from \cref{thm:c-Bubble_Line}. We have $\varphi(z)=\beta z$ and $\psi(z)=\beta z+\alpha$ for all $z\in[0,1]$. Let $\Paral^{\alpha,\beta}=\mathscr{R}^{\varphi,\psi}$, and let $\rho^{\alpha,\beta}=\upsilon\star\zeta_1^{\Paral^{\alpha,\beta}}$, where $\upsilon$ is the uniform measure on $[0,1]^2$. The height function of $\rho^{\alpha,\beta}$ is given by 
\[\h_{\rho^{\alpha,\beta}}(x,y)=\min_{0\leq\gamma\leq 1}((1-\gamma)y+\hh_1(x,y)),\] where we write $\hh_1$ instead of $\hh_1^{\varphi,\psi}$. To deduce \cref{thm:c-Bubble_Line} from \cref{thm:BubbleGeneral}, we must show that $\rho^{\alpha,\beta}=\nu^{\alpha,\beta}$. That is, we must show that 
\begin{equation}\label{eq:parallelogram}
\h_{\nu^{\alpha,\beta}}(x,y)=\min_{0\leq\gamma\leq 1}\left((1-\gamma)y+\hh_1(x,\gamma)\right)
\end{equation} for every $(x,y)\in[0,1]^2$. If $x=1$ or $y=0$, then $\h_{\nu^{\alpha,\beta}}(x,y)$ and $\h_{\rho^{\alpha,\beta}}(x,y)$ are clearly both equal to $0$, so we will assume in what follows that $x<1$ and $y>0$. 

In order to reduce the number of cases we must consider, we will employ the symmetry discussed in \cref{rem:symmetry2}. Note that the parallelogram $\Paral^{\alpha,1-\beta}$ is a translation of the parallelogram obtained by reflecting $\Paral^{\alpha,\beta}$ through the line $\mathfrak L_{1/2}$. It follows that $\rho^{\alpha,1-\beta}=\widehat{\rho^{\alpha,\beta}}$. Because $\nu^{\alpha,1-\beta}=\widehat{\nu^{\alpha,\beta}}$, we only need to check \eqref{eq:parallelogram} when $y\leq 1-\beta$.  

Fix $(x,y)\in[0,1]^2$ with $x<1$ and $0<y\leq 1-\beta$. Note that \[\hh_1(x,\gamma)=\min\{\max\{0,\gamma-x,(1-\beta)(\gamma-x)+\alpha\},1-x,\gamma\}.\]
Let $Q(\gamma)=(1-\gamma)y+\hh_1(x,\gamma)$. View $\hh_1(x,\gamma)$ and $Q(\gamma)$ as functions of $\gamma$. It is straightforward to compute that $\hh_1(x,0)=0$ and $\hh_1(x,1)=1-x$. Therefore, there exist $q,q'\in[0,1]$ with $q\leq q'$ such that 
\[\{\gamma\in[0,1]:\hh_1(x,\gamma)=0\}=[0,q]\quad\text{and}\quad\{\gamma\in[0,1]:\hh_1(x,\gamma)=1-x\}=[q',1].\] Moreover, the function $Q$ is decreasing on the open intervals $(0,q)$ and $(q',1)$, and it is non-decreasing on $(q,q')$. This implies that the minimum value of $Q(\gamma)$ is attained when $\gamma=q$ or when $\gamma=1$. In other words, \[\min_{0\leq \gamma\leq 1}((1-\gamma)y+\hh_1(x,\gamma))=\min\{(1-q)y+\hh_1(x,q),\hh_1(x,1)\}=\min\{(1-q)y,1-x\}.\] 
Hence, our goal is to show that $\h_{\nu^{\alpha,\beta}}(x,y)=\min\{(1-q)y,1-x\}$. 

Suppose first that $x\geq\alpha/(1-\beta)$. Then it is straightforward to verify that $q=x-\alpha/(1-\beta)$. If \[y\leq 1-\frac{\alpha}{(1-\beta)(1-x)+\alpha}\] (meaning $(x,y)$ is in the region $\mathscr U_{\searrow}^{\alpha,\beta}$ defined in \eqref{eq:UU1}), then we have \[\h_{\nu^{\alpha,\beta}}(x,y)=(1-x+\alpha/(1-\beta))y=(1-q)y\] and ${(1-x+\alpha/(1-\beta))y\leq 1-x}$, which completes the proof. On the other hand, if \[y\geq 1-\frac{\alpha}{(1-\beta)(1-x)+\alpha}\] (meaning $(x,y)$ is above the region $\mathscr U_{\searrow}^{\alpha,\beta}$), then we have \[\h_{\nu^{\alpha,\beta}}(x,y)=1-x\] and ${(1-x+\alpha/(1-\beta))y\geq 1-x}$, which completes the proof. 

Now suppose $x\leq\alpha/(1-\beta)$. In this case, we have $\h_{\nu^{\alpha,\beta}}(x,y)=\min\{y,1-x\}$, so it suffices to show that $q=0$. Thus, we must show that $\hh_1(x,\gamma)>0$ for every $\gamma\in(0,1]$. Choose some $\gamma\in(0,1]$. We have 
\[
\max\{0,\gamma-x,(\gamma-x)(1-\beta)+\alpha\}\geq (\gamma-x)(1-\beta)+\alpha 
\geq\gamma(1-\beta) 
\geq\gamma y 
>0. 
\] 
Hence, 
\[\hh_1(x,\gamma)=\min\{\max\{0,\gamma-x,(\gamma-x)(1-\beta)+\alpha\},1-x,\gamma\}>0.\]

\subsection{Rectangle Shapes and Memory}\label{subsec:diary}  
Let us fix $\beta\in[0,1]$ and preserve the notation from \cref{thm:rectangle}. We have \[\varphi(z)=\begin{cases}
    0 & \text{if $0\leq z\leq1-\beta$}; \\
    z-1+\beta & \text{if $1-\beta\leq z\leq 1$} 
\end{cases}\quad\text{and}\quad \psi(z)=\begin{cases}
    z & \text{if $0\leq z\leq \beta$}; \\
    \beta & \text{if $\beta\leq z\leq 1$}  
\end{cases}\] for all $z\in[0,1]$. Let $\Rect^{\beta}=\mathscr{R}^{\varphi,\psi}$, and let $\kappa^{\beta}=\upsilon\star\zeta_1^{\Rect^{\beta}}$, where $\upsilon$ is the uniform measure on $[0,1]^2$. The height function of $\kappa^{\beta}$ is given by 
\[\h_{\kappa^{\beta}}(x,y)=\min_{0\leq\gamma\leq 1}((1-\gamma)y+\hh_1(x,y)),\] where we write $\hh_1$ instead of $\hh_1^{\varphi,\psi}$. To deduce \cref{thm:rectangle} from \cref{thm:BubbleGeneral}, we must show that $\kappa^{\beta}=\nu^{\beta(1-\beta),\beta}$. That is, we must show that 
\begin{equation}\label{eq:rectangle}
\h_{\nu^{\beta(1-\beta),\beta}}(x,y)=\min_{0\leq\gamma\leq 1}\left((1-\gamma)y+\hh_1(x,\gamma)\right)
\end{equation} for every $(x,y)\in[0,1]^2$. 

Just as in \cref{subsec:parallelogram}, we can use the symmetry from \cref{rem:symmetry2}. Because $\Rect^{1-\beta}$ is a translation of the rectangle obtained by reflecting $\Rect^\beta$ through the line $\mathfrak L_{1/2}$, we have $\kappa^{1-\beta}=\widehat{\kappa^\beta}$. Since $\nu^{\beta(1-\beta),1-\beta}=\widehat{\nu^{\beta(1-\beta),\beta}}$, we only need to check \eqref{eq:rectangle} when $x\leq\beta$. 

Fix $(x,y)\in[0,1]^2$ with $x\leq \beta$. Then $(x,y)$ is either contained in or below the region $\mathscr{U}_{\nwarrow}^{\beta(1-\beta),\beta}$, and it follows from the definition of $\nu^{\beta(1-\beta),\beta}$ that \[\h_{\nu^{\beta(1-\beta)}}(x,y)=\min\{y,(\beta-x)y+1-\beta\}.\] 
Let $Q(\gamma)=(1-\gamma)y+\hh_1(x,\gamma)$. Note that 
\[\hh_1(x,\gamma)=\min\{\max\{0,\gamma-x,\gamma-\varphi(\gamma)\},1-x,\gamma\}.\] 
If $0\leq\gamma\leq 1-\beta$, then $\varphi(\gamma)=0$ and $\gamma\leq 1-x$, so 
$\hh_1(x,\gamma)=\gamma$. If $1-\beta\leq\gamma\leq x+1-\beta$, then $\varphi(\gamma)=\gamma-1+\beta$, so $\hh_1(x,\gamma)=1-\beta$. If $x+1-\beta\leq\gamma\leq 1$, then $\varphi(\gamma)=\gamma-1+\beta$, so $\hh_1(x,\gamma)=\gamma-x$. It follows that $Q$ is non-decreasing (with slope $1-y$) on the open interval $(0,1-\beta)$, non-increasing (with slope $-y$) on the open interval $(1-\beta,x+1-\beta)$, and non-decreasing on the open interval $(x+1-\beta,1)$. Therefore, the minimum value of $Q(\gamma)$ is attained when $\gamma=0$ or when $\gamma=x+1-\beta$. In other words, 
\begin{align*} 
\min_{0\leq\gamma\leq 1}((1-\gamma)y+\hh_1(x,\gamma))&=\min\{Q(0),Q(x+1-\beta)\} \\ 
&=\min\{y,(\beta-x)y+1-\beta\} \\ &=\h_{\nu^{\beta(1-\beta),\beta}}(x,y).
\end{align*} This completes the proof of \eqref{eq:rectangle} and, hence, the proof of \cref{thm:rectangle}. 

Let us now consider the relationship between the permuton\footnote{We call $\nu^{\beta(1-\beta),\beta}$ a \dfn{Paracanthurus permuton} because it resembles two fins.} $\nu^{\beta(1-\beta),\beta}$ and the model of Dory's memory described at the end of \cref{subsec:intro_Bubble}. Recall that Dory is capable of storing $k=\left\lfloor(1-\beta) n\right\rfloor$ in her mind, where $0<\beta<1$. Let $q_1,\ldots,q_k$ be the relevance factors of the facts that Dory knows on Day $0$, listed in an arbitrary order. For $k+1\leq j\leq n$, let $q_j$ be the relevance factor of the fact that Dory learns on Day $j-k$. Thus, $q_n,q_{n-1},\ldots,q_1$ is a sequence of independent real numbers chosen uniformly at random from $[0,1]$. We may assume that the numbers in this sequence are all distinct since this happens with probability $1$. Therefore, the sequence $q_n,q_{n-1},\ldots,q_1$ has the same relative order as some permutation $w^{(0)}$. The entry $w^{(0)}(i)\in[n]$ corresponds to a certain fact whose relevance factor is the number $q_{n+1-i}\in[0,1]$. It is well known that the distribution of $w^{(0)}$ is the uniform distribution on $\SSS_n$. 

For $1\leq i\leq n-k$, we consider the word ${\bf v}_i=(n-i-k+1,n-i-k+2,\ldots,n-i)$ of length $k$. Let $w^{(j)}=(\tau_{{\bf v}_j}\circ\tau_{{\bf v}_{j-1}}\circ\cdots\circ\tau_{{\bf v}_1})(w^{(0)})$. Observe that the concatenated word ${\bf v}_1{\bf v}_2\cdots{\bf v}_{n-k}$ is a word ${\bf w}(\sS^{(n)})$ corresponding to a linear extension of a $k\times (n-k)$ rectangle shape $\sS^{(n)}$. Thus, $w^{(n-k)}$ has the same distribution as the permutation $u_n$ in \cref{thm:rectangle}. 

The fact that Dory learns on Day $1$ corresponds to $w^{(0)}(n-k)$. When we apply $\tau_{{\bf v}_1}$ to $w^{(0)}$, the smallest element of $w^{(0)}([n-k,n])$ moves all the way to the right end of the permutation. This smallest element, which is $w^{(1)}(n)$, corresponds to the fact that Dory forgets on Day $1$. When we apply $\tau_{{\bf v}_2}$ to the permutation $w^{(1)}$, the smallest element of $w^{(1)}([n-k-1,n-1])$ moves to the right until settling in position $n-1$. This element, which is $w^{(2)}(n-1)$, corresponds to the fact that Dory forgets on Day $2$. Note that the element $w^{(2)}(n)=w^{(1)}(n)$ corresponds to the fact that Dory forgot on Day $1$. Continuing in this manner, we eventually find that the elements $w^{(n-k)}(n),w^{(n-k)}(n-1),\ldots,w^{(n-k)}(k+1)$ correspond to the facts that Dory forgets on Days $1,2,\ldots,n-k$, respectively. Recall that we denote the relevance factors of these facts by $r_1,r_2,\ldots,r_{n-k}$. Hence, for large $n$, the plot of the points $(1-(i-1)/n,r_i)$ for $1\leq i\leq n-k$ closely approximates the portion of the plot of $w^{(n-k)}$ lying in the region $[(k+1)/n,1]\times[0,1]$. Consequently, it also approximates the portion of the plot of the random permutation $u_n$ lying in this region. \cref{thm:rectangle} tells us that $\pi_{u_n}$ is close (in the weak topology) to the permuton $\nu^{\beta(1-\beta),\beta}$. 

\section{Concluding Remarks}\label{sec:conclusion} 

\subsection{Fluctuations}
In \cref{thm:fluctuations}, we computed the asymptotic fluctuations of the height functions of permutons arising from Demazure products of random pipe dreams in order-convex shapes. We have also discussed several other permuton limit results in this article; it would be very interesting to understand the fluctuations of the height functions in these other settings as well. As a possible approach, we seek a method that, given explicit permutons $\mu$ and $\nu$, would allow one to compute the fluctuations of $\h_{u\star v}(x,y)$ when $u$ is a $\mu$-random permutation of size $n$ and $v$ is a $\nu$-random permutation of size $n$. 

In a similar vein, we again highlight the problem of finding a more refined version of \cref{thm:random_star_random} (or \cref{thm:random_star_random_patterns}).   

\subsection{Doppelg\"angers} 
In \cref{rem:doppelgangers}, we noted the mysterious coincidence of two permuton limits arising from different shapes in $\RR$, a rectangle and a non-rectangular parallelogram. Specifically, we found that $\upsilon\star\zeta_1^{\Paral^{\beta(1-\beta),\beta}}=\upsilon\star\zeta_1^{\Rect^\beta}$, where $\upsilon$ is the uniform measure on $[0,1]^2$ and $\Paral^{\beta(1-\beta),
\beta}$ and $\Rect^\beta$ are as defined in \cref{subsec:parallelogram,subsec:diary}. It would be interesting to have a more conceptual explanation of this coincidence or additional examples of this phenomenon. 

\subsection{Stripes} 
Several of the plots of large random permutations depicted in this article exhibit noticeable ``stripes'' (see \cref{fig:rectangle_example_plot,fig:trapezoid_example_plot,fig:gems,fig:L-shape,fig:hole,fig:more}.) It would be very interesting to prove a rigorous statement explaining this phenomenon. In a related vein, one could study local limit properties of these large random permutations in the sense that Borga has recently developed~\cite{Borga}. 

\subsection{More Pictures} 
Many of the permutons discussed in this article are of the form $\mu\star\zeta_p^\DD$, where $p\in(0,1]$, $\mu$ is either the identity permuton or the uniform permuton, and $\DD\in\RR$. We end with \cref{fig:more}, which provides an array of plots of permutations in $\SSS_{2000}$ that approximate some other permutons of this form. We use the notation $\Paral^{\alpha,\beta}$ and $\Rect^\beta$ from \cref{subsec:parallelogram,subsec:diary}. We also let $\ULT\in\RR$ be the limit of scaled staircase shapes (i.e., the region $\mathscr{R}^{\varphi,\psi}$ with $\varphi(z)=0$ and $\psi(z)=z$), and we let $\LLT\in\RR$ be the region obtained by reflecting $\ULT$ through the horizontal line $\{(x,1/2):x\in\mathbb R\}$.  

\begin{figure}[h]
  \begin{center}
  \includegraphics[height=5cm]{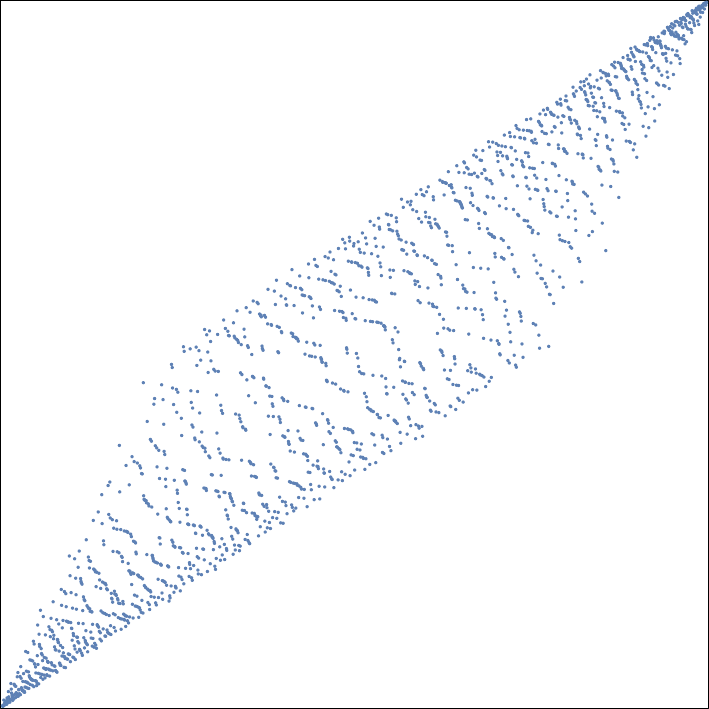}\quad\includegraphics[height=5cm]{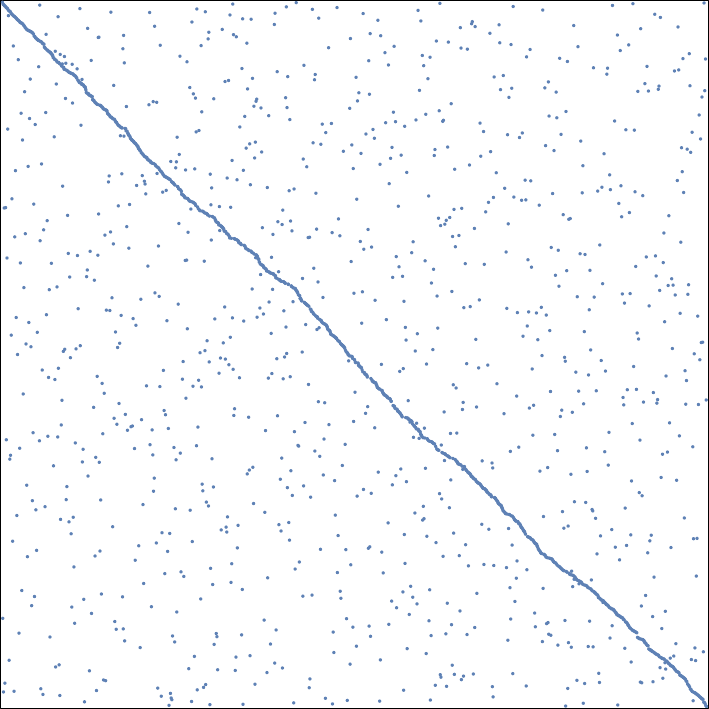}\quad\includegraphics[height=5cm]{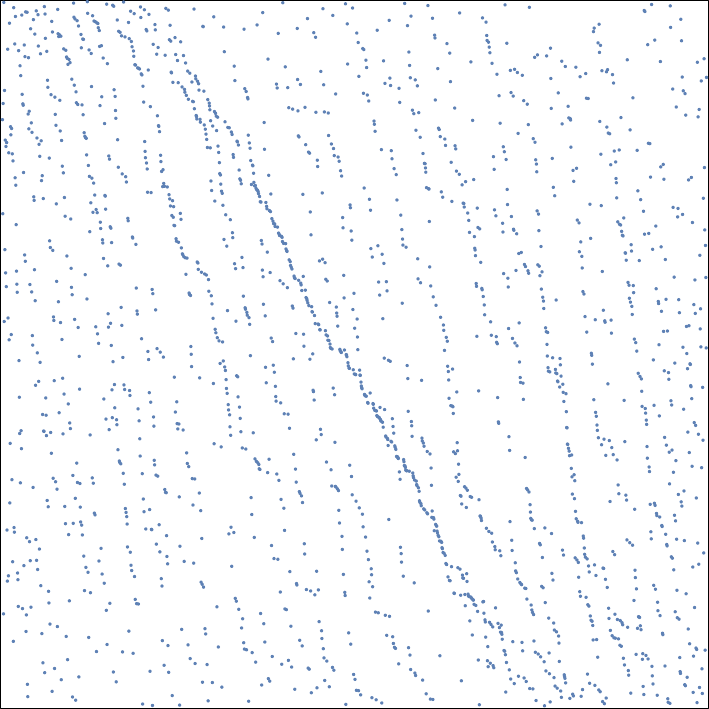} \\ \vspace{0.3cm}
    \includegraphics[height=5cm]{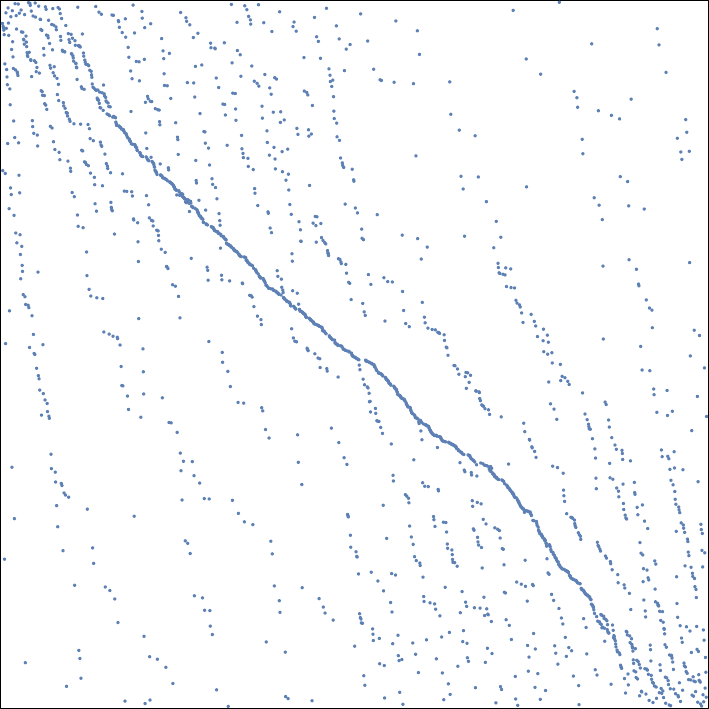}\quad\includegraphics[height=5cm]{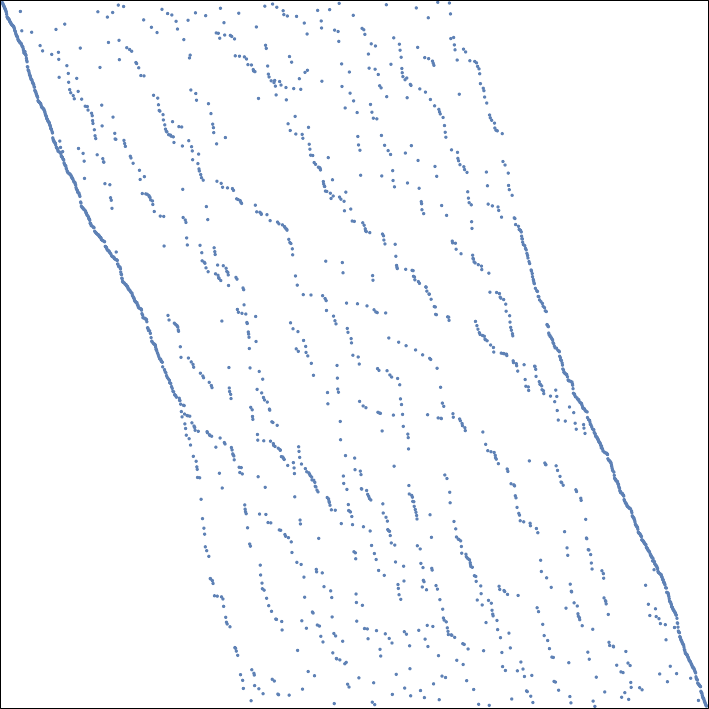}\quad\includegraphics[height=5cm]{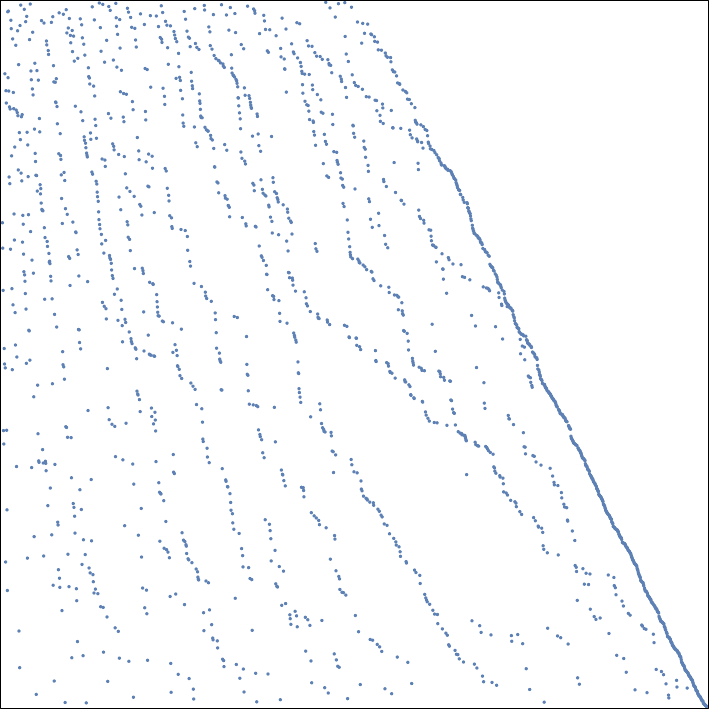}
  \end{center}
\caption{Plots of permutations in $\SSS_{2000}$ that approximate the permutons $\zeta_{2/3}^{\LLT},\,\, \upsilon\star\zeta_1^{\LLT},\,\, \upsilon\star\zeta_{2/3}^{\LLT},\,\, \upsilon\star\zeta_{5/6}^{\Rect^{1/2}},\,\, \upsilon\star\zeta_{5/6}^{\Paral^{1/4,1/2}},\,\, \upsilon\star\zeta_{1/2}^{\ULT}$. }\label{fig:more}  
\end{figure}

\section*{Acknowledgments}
The author was supported by the National Science Foundation under Award No.\ 2201907 and by a Benjamin Peirce Fellowship at Harvard University. He thanks Elise Catania, Greta Panova, Leo Petrov, Derek Liu, Mikhail Tikhonov, and Katherine Tung for helpful discussions.

\end{document}